\documentclass{amsart}
\usepackage[PostScript=dvips]{diagrams}
\usepackage{amsmath,amstext,amscd, amssymb}
\usepackage{enumerate}
\usepackage{epsfig, psfrag, color, multicol, graphicx}
\theoremstyle{plain}

\newtheorem{thm}{Theorem}[section]
\newtheorem{lemma}[thm]{Lemma}

\newtheorem{cor}[thm]{Corollary}
\newtheorem{prop}[thm]{Proposition}

\newtheorem{conjecture}[thm]{Conjecture}

\theoremstyle{remark}
\newtheorem{ex}[thm]{Example}

\newtheorem{rem}[thm]{Remark}

\theoremstyle{definition}
\newtheorem{defn}[thm]{Definition}

\newtheorem{remark}[thm]{Remark}

\newtheorem{Open questions}[thm]{Open questions}
\newtheorem{Open question}[thm]{Open question}
\newtheorem{Open problems}[thm]{Open problems}
\newtheorem{Open problem}[thm]{Open problem}

\newenvironment{enumi}{\begin{enumerate}[\small\upshape(i)]}{\end{enumerate}}
\newenvironment{enum1}{\begin{enumerate}[\small\upshape(1)]}{\end{enumerate}}

\newcommand{\tc}[2]{\textcolor{#1}{#2}}
\definecolor{dmagenta}{rgb}{.5,0,.5} 
\definecolor{dred}{rgb}{.5,0,0} 
\definecolor{dgreen}{rgb}{0,.5,0} 
\definecolor{blue}{rgb}{0,0,0.5} 
\definecolor{black}{rgb}{0,0,0} 
\definecolor{vdgreen}{rgb}{0,.3,0} 
\definecolor{vdred}{rgb}{.3,0,0} 
\definecolor{red}{rgb}{1,0,0} \newcommand{\red}[1]{\tc{red}{#1}}

\newcommand{\Z}{\mathbb{Z}}

\newcommand{\bG}{\mathbf{G}}

\def\wh{\widehat}

\def\emp{\nothing}

\def\zz{\mathbb Z}
\def\zzz{{{\zz_2}}}
\def\nn{\mathbb N}

\def\rr{\mathbb R}

\def\ff{{\rm \textbf{F}}}

\def\Ga{\Gamma}

\def\Si{\Sigma}
\def\ga{\gamma}

\def\ep{\epsilon}
\def\al{\alpha}
\def\be{\beta}
\def\om{\omega}

\def\ka{\vk}

\def\vk{\varkappa}

\def\cG{\mathcal G}

\def\ssu{\subset}

\def\wt{\widetilde}
\def\<{\langle}
\def\>{\rangle}

\def\SL{ {\text {\rm SL} } }

\DeclareMathOperator{\Aut}{Aut}
\DeclareMathOperator{\diam}{diam}

\def\Ups{\Upsilon}

\def\vt{\vartheta}

\def\rK{{\text{\rm {K}}}}

\def\ii{{\text{\textbf{\em i}}}}

\def\0{{\mathbf 0}}

\def\nothing{\varnothing}

\def\.{\hskip.06cm}
\def\ts{\hskip.03cm}

\def\length{{\text {\rm $\ell$}}}

\def\cayley{{\text {\rm Cayley} }}
\def\Sym{{\text {\rm Sym} }}

\def\nts{\hskip-.04cm}

\newcommand{\F}{\ff}

\newcommand{\la}{\langle}
\newcommand{\ra}{\rangle}
\DeclareMathOperator{\PSL}{PSL}

\def\sml{{\preccurlyeq}}
\def\grt{{\succcurlyeq}}

\newcommand{\toto}{{\ts\twoheadrightarrow\ts}}
\DeclareMathOperator{\bp}{{{{ \bigotimes \ts }}}}
\newcommand{\cff}{\vartheta}
\def\op{{{{ \ts \otimes \ts }}}}

\def\gaga{{\Ups}}
\def\Upsi{{\phi}}

\def\Laoi{{\Lambda^i_{\om}}}
\def\Laom{{\Lambda^M_{\om}}}

\def\lao{{\triangle}}
\def\laoi{{i}}

\def\Ds{{ {\mathcal{D}_- }}}
\def\Dg{{ {\mathcal{D}}_+ }}
\def\Deq{{ {\mathcal{D}_\circ }}}

\def\Dsi{{ {\mathcal{D}^i_- }}}
\def\Dgi{{ {\mathcal{D}^i_+ }}}
\def\Deqi{{ {\mathcal{D}^i_\circ }}}

\def\vs{\varrho}
\def\rT{{\mathbf{T}}}

\newcommand{\ccc}{c}
\newcommand{\ddd}{{\text{\rm D}}}

\newcommand{\UUU}[2]{{\text{\rm U}}_{#2}({#1})}
\newcommand{\LLL}[2]{{\text{\rm L}}_{#2}({#1})}
\newcommand{\NNN}[2]{{\text{\rm N}}_{#2}({#1})}

\title{Groups of Oscillating Intermediate Growth}

\author[Martin~Kassabov]{ \ Martin~Kassabov$^\star$}

\author[Igor~Pak]{ \ Igor~Pak$^\dagger$}

\thanks{\thinspace ${\hspace{-.45ex}}^\star$School of Mathematics, University of Southampton, 
Southampton, SO17 1BJ, UK, \ts and Department of Mathematics, Cornell University, Ithaca, NY 14853, USA; \ts
Email: \ts \texttt{martin.kassabov@gmail.com}}

\thanks{\thinspace ${\hspace{-.45ex}}^\dagger$Department of Mathematics,
UCLA, Los Angeles, CA 90095, USA; \ts Email: \ts \texttt{pak@math.ucla.edu}}

\date \today

\begin{document}

\begin{abstract}
We construct an uncountable family of finitely generated groups of
intermediate growth, with \emph{growth functions} of new type.
These functions can have large oscillations between lower and
upper bounds, both of which come from a wide class of functions.
In particular, we can have growth oscillating between~$e^{n^\al}$
and any prescribed function, growing as rapidly as desired.
Our construction is built on top of any of the Grigorchuk
groups of intermediate growth, and is a variation on the
limit of permutational wreath product.
\end{abstract}


\maketitle

\section{Introduction}

The growth of finitely generated groups is a beautiful subject
rapidly developing in the past few decades.  With a pioneer
invention of groups of \emph{intermediate growth} by
Grigorchuk~\cite{Gri2} about thirty years ago, it became
clear that there are groups whose growth is given by difficult
to analyze function.  Even now, despite multiple improvements,
much about their growth functions remains open
(see~\cite{Gri-one,Gri5}), with the sharp bounds constructed
only this year in groups specifically designed for that
purpose~\cite{BE}.

In the other direction, the problem of characterizing growth
functions of groups remains a major open problem, with only
partial results known.  Part of the problem is a relative
lack of constructions of intermediate growth groups, many
of which are natural subgroups of $\Aut(\rT_k)$, similar to
the original Grigorchuk groups in both the structure and analysis.
In this paper we propose a new type of groups of intermediate
growth, built by combining the action of Grigorchuk groups
on $\rT_k$ and its action on a product of copies of certain
finite groups $H_i$.  By carefully controlling groups~$H_i$,
and by utilizing delicate expansion results, we ensure
that the growth oscillates between two
given functions. Here the smaller function is controlled by
a Grigorchuk group, and the larger function can be
essentially \emph{any} sufficiently rapidly growing function
(up to some technical condition).  This is the first result
of this type, as even the simplest special cases could not
be attained until now (see the corollaries below).

\smallskip

Our main result (Main Theorem~\ref{t:main}) is somewhat technical
and is postponed until the next section.  Here in the introduction
we give a rough outline of the theorem, state several corollaries,
connections to other results, etc.  For more on history of the
subject, general background and further references see
Section~\ref{s:fin}.

\smallskip

For a group $\Ga$ with a generating set $S$, let $\ga_\Ga^S(n) = |B_{\Ga,S}(n)|$,
where $B_{\Ga,S}(n)$ is the set of elements in~$\Ga$ with word length~$\le n$.
Suppose $f_1,f_2,g_1,g_2: \nn \to \nn$
monotone increasing subexponential integer functions which satisfy
\begin{equation}
\label{eq:ast}
\tag{$\ast$}
f_1\. \grt \. f_2 \.\grt \. g_1 \. \grt \. g_2 \. = \. \ga^S_{\bG_\om}, \ \ \.
\text{where} \ \,\, \bG_\om=\<S\> \ \,
\text{is a \emph{Grigorchuk group}~\cite{Gri3}\ts.}
\end{equation}
Roughly, the \textbf{Main Theorem} states that under further technical
conditions strengthening~(\ref{eq:ast}), there exists a finitely
generated group $\Ga$ and a generating set $S$,
with growth function $h(n) = \ga_{\Ga}^S(n)$, such that
$g_2(n) \. < \. h(n) \. < \. f_1(n)$ and $h(n)$ takes values in the
intervals $[g_2(n),g_1(n)]$ and $[f_2(n),f_1(n)]$
infinitely often.
We illustrate the theorem in Figure~\ref{f:graph}.

\smallskip

\begin{figure}[hbt]
\begin{center}
\psfrag{a}{\small $f_1$}
\psfrag{b}{\small $f_2$}
\psfrag{c}{\small $g_1$}
\psfrag{d}{\small $g_2$}
\psfrag{h}{$h$}
\psfrag{n}{$n$}
\psfrag{1}{\hskip-.16cm $1$}
\psfrag{0}{\hskip-.16cm $.5$}
\epsfig{file=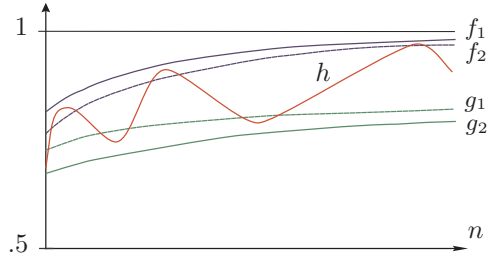,width=6.2cm}
\end{center}
\caption{The graph of $\log_n \log$ of functions $f_1,f_2,g_1,g_2$
   and $h$, as in the Main Theorem.}
\label{f:graph}
\end{figure}

Stated differently, the Main Theorem implies that one can construct groups
with specified growth which is oscillating within a certain range, between,
$\exp\bigl(n^{\al(n)}\bigr)$ where $\al(n)$ is bounded from below by a
constant, and a function which converges to {\small $(1-)$} as rapidly as desired.
Of course, it is conjectured that $\al(n)\ge 1/2$ for all groups of
intermediate growth~\cite{Gri-one} and~$n$ large enough
(cf.~Subsection~\ref{ss:fin-growth}).

\smallskip

To get some measure of the level of complexity of this result,
let us state a corollary of independent interest.  Here we take the
\emph{first Grigorchuk group}~$\bG$ (cf.~Subsection~\ref{ss:fin-gri}),
and omit both~$f_1$ and $g_2$, taking~$g_1$ to be slightly greater
that the best known upper bound for the growth of~$\bG$.

\begin{cor}[Oscillating Growth Theorem]
\label{cor:osc-growth}
For every  increasing function $\mu: \nn\to \rr_+$, $\mu(n) = o(n)$,
there exists a finitely generated group $\Ga$ and a
generating set $\<S\> = \Ga$, such that
\begin{align*}
\ga_{\Ga}^S(n) & \, < \, \exp \bigl(n^{4/5}\bigr) \quad \ \ \text{for infinitely many} \ \ n\in \nn\ts,
\ \ \text{and} \\
\ga_{\Ga}^S(n') & \, > \, \exp\bigl(\mu(n')\bigr) \quad \ \, \text{for infinitely many} \ \ n'\in \nn\ts.
\end{align*}
\end{cor}

\medskip

The corollary states that the growth of~$\Ga$ is both
\emph{intermediate} (i.e., super polynomial and subexponential),
and \emph{oscillating} between two growth functions which may
have different asymptotic behavior.  In fact, $\mu(n)$ can be as
close to linear function as desired, so for example one
can ensure that the ball sizes are $\le \exp(n^\al)$ for some~$n$,
and $\ge \exp(n/\log\log\log n)$ for other $n$,
both possibilities occurring infinite often.  Let us now
state a stronger version of the upper bound in the same setting.

\begin{cor}[Oscillating growth with an upper bound] \label{cor:osc-upper}
For $\mu(n) = n^{\al} \log^\be n$ with $5/6 < \al < 1$, or $\al =1$ and
$\be < 0$, there exist a finitely generated group $\Ga$ and a
generating set $\<S\> = \Ga$,  such that
$$
\ga_{\Ga}^S(n) \. < \. \exp(n^{4/5}),
\quad
\ga_{\Ga}^S(n') \. > \.  \exp(\mu(n'))
\quad \text{for infinitely many} \ \ n,~n'\in \nn\ts,
$$
$$
\mbox {and } \ \, \ga_{\Ga}^S(m) \, < \, \ts
\exp\bigl(\mu(\text{\rm \small 9}\ts m)\bigr)
\quad  \text{for all sufficiently large } \ m\in \nn\ts.
$$
\end{cor}

In other words, $\ga_\Ga(n_i)$ has the same asymptotics as $\ts e^{\mu(n_i)}$,
for a certain infinite subsequence~$\{n_i\}$.  Note a mild
restriction on~$\al$, which is a byproduct of our technique
(see other examples in the next section).

\smallskip

Now, the Main Theorem and the corollaries are related to
several other results.  On the one hand, the growth
of balls in the first Grigorchuk group~$\bG$
is bounded from above and below by
$$
\exp (n^\be) < \ga_{\bG}^S(n) < \exp (n^\al) \ \,
\text{for all generating sets~$S$, }
$$
integers~$n$ large enough, and where $\al=0.7675$ and $\be = 0.5207$.
Since this $\al$ is the smallest available upper bound for any known
group of intermediate growth (see Subsection~\ref{ss:fin-gri}), this
explain the lower bound in the corollaries
(in fact, the power $4/5$ there can be lowered to any $\al'> \al$).

For the upper bound, a result by Erschler~\cite{Ers2}
states that there is a group of intermediate growth,
such that $\ga_{\Ga}^S(n) > f(n)$ for all $n$ large
enough.  This result does not specify exactly the
asymptotic behavior of the growth function $\ga_{\Ga}^S(n)$,
and is the opposite extreme when compared to the Oscillating
Growth Theorem, as here \emph{both} the upper and lower bounds
can be as close to the exponential function as desired.
This also underscores the major difference with our main
theorem, as in this paper we emphasize the upper bounds
on the growth, which can be essentially any subexponential
function.

Combined with this Erschler's result, the Main Theorem
states that one can get an oscillating growth phenomenon
as close to the exponential function as desired.  For example,
Erschler showed in~\cite{Ers1}, that a certain group
Grigorchuk group~$\bG_\om$ has growth between
$g_2=\exp(n/\log^{2+\ep} n)$ and $g_1= \exp(n^{1-\ep}/\log n)$,
for any $\ep>0$.  The following result is a 
special case of the Main Theorem applied to this group~$\bG_\om$.

\begin{cor}\label{cor:osc-ers}
Fix $\ep >0$. Define four functions: \,
$g_2(n) = e^{n/\log^{2+\ep} n}$\., \ $g_1(n) = e^{n/\log^{1-\ep} n}\.,$
$$
f_2(n) = e^{n/\log \log n}\.,
\. \ \ \text{and} \ \,\. f_1(n) =
e^{n\sqrt{\log\log\log n}/\log\log n}\..
$$
Then there exists a finitely generated group $\Ga$ and a
generating set $\<S\> = \Ga$, with growth function
$h(n) = \ga_{\Ga}^S(n)$, such that:
$$
g_2(n) < h(n) < f_1(n)\ts, \ \
\text{for all $n$ large enough, and}
$$
$$
h(m) < g_1(m)\ts,  \ \. h(m') > f_2(m') \ts, \ \
\text{for infinitely many $m, m' \in \nn$\ts.}
$$
\end{cor}

Of course, here the functions~$f_1, f_2$ are chosen somewhat
arbitrarily, to illustrate the power of our Main Theorem.

\smallskip

Let us say now a few words about the \emph{oscillation phenomenon}.
In a recent paper~\cite{Bri2}, Brieussel shows that there
is a group~$\Ga$ of intermediate growth, such that
$$
\liminf \. \frac{\log \log \ga_{\Ga}(n)}{\log n} \. = \al_-
\quad
\mbox{and}
\quad
\limsup \. \frac{\log \log \ga_{\Ga}(n)}{\log n} \. = \al_+,
$$
for any fixed $\al=0.7675\leqslant \al_- \leqslant \al_+ \leqslant 1$.
This result is somewhat weaker than our Main Theorem when it comes to
the range of asymptotics of upper limits, but is stronger in a sense
that the lower limits can be prescribed in advance, and $\al_+-\al_-$
can be as small as desired (cf.~Example~\ref{ex:main-growth1}).
Since Brieussel uses groups different from~$\bG_\om$,
the Main Theorem cannot use them as an input.  We postpone
further discussion of this until Section~\ref{s:gen} (see also
Subsection~\ref{ss:fin-comparison}).

\smallskip

A starting point for the construction in the proof of the Main Theorem
(Theorem~\ref{t:main}), is a sequence of finite groups~$G_i$ with
generating sets~$S_i$, such that the growth of small balls in
$X_i = \cayley(G_i,S_i)$ is roughly~$\ga_{\bG_\om}$, but the
diameter of Cayley graphs~$X_i$ is close to logarithmic, i.e.,
the growth of large balls is almost exponential.
These groups and generating sets can be combined into an
infinite group~$\Gamma$ and generating set~$S$, such that for
certain~$n$, balls $\ga_{\Ga}^S(n)$ behave as small balls
in~$X_i$, while for other values of $n$, these balls behave
as large balls in~$X_i$, and thus have almost exponential size.
This behavior implies that size of these balls oscillates as in
the theorem.

\smallskip

The rest of the paper is structured as follows.  We begin with
the statement of the main theorem (Section~\ref{s:main}), where
also give examples and a very brief outline of the proof idea.
We continue  with basic definitions and notations in
Section~\ref{s:def}.  We then
explore graph and group limits in Section~\ref{sec_limits}.
We explore the Grigorchuk groups $\bG_\om$ in Section~\ref{s:gri},
giving some preliminary technical results, which we continue in
Section~\ref{s:technical}. We then prove the
Oscillating Growth Theorem (Corollary~\ref{cor:osc-growth})
in Section~\ref{s:proof}, as without functions $g_1$ and~$f_2$
the result is technically easier to obtain.  We then prove the
Main Theorem~\ref{t:main} in Section~\ref{s:control}.
A key technical result (Main Lemma~\ref{l:postponed}) is postponed
until Section~\ref{s:lemma-postponed}, while further
generalizations are presented in Section~\ref{s:gen}.
We conclude with final remarks and open problems in
Section~\ref{s:fin}.

\smallskip

\section{The Main theorem}
\label{s:main}

\subsection{The statement}  \label{ss:main-statement}
We begin with two technical definitions.

\begin{defn}
A function $f: \nn \to \rr$ is called \emph{admissible} if $f(n)$ is increasing,
subexponential, and the ratio
$$
\frac{n}{\log f(n)} \quad \text{is increasing.}
$$
\end{defn}


\begin{defn}
If $f$ is an admissible function, define a function $f^*: \rr \to \rr$ as the
following inverse function:
$$
f^*(z) \ts = \ts \Phi^{-1}(z),  \quad \mbox{where} \ \
\Phi(x) \. = \.  \frac{x}{\log f(x)}\ts.
$$
\end{defn}


\smallskip

\begin{thm} [Main theorem] \label{t:main}
Let $f_1,f_2, g_1 , g_2: \nn \to \nn$ be functions
which satisfy the following conditions:

\begin{enumi}
\item \label{t:main:con:admissible}
$f_1$ and $f_2$ are admissible,

\item \label{t:main:con:square}
$f_2(n)^3 = o\bigl(f_1(n)\bigr)$,

\item \label{t:main:con:increasing}
$f_1/g_1$ is increasing,

\item \label{t:main:con:restriction}
$g_2(n) \le \ga_{\bG_\om}(n)$, where $\bG_\om$ is
a Grigorchuk group of intermediate growth,

\item \label{t:main:con:lower}
$\ga_{\bG_\om}(n) = o\bigl(g_1(n)\bigr)$,

\item \label{t:main:con:growth}
$\displaystyle
\exp\left[
\frac{\log g_1(n)}{Cn^2} \,
f_1^*\left(\frac{n}{C\log g_1(n)}\right)
\right] \.
> \.
\frac{C\ts f_2^*(Cn)}{n^2}$, \ts
for all $C>0$ and $n=n(C)$ sufficiently large.
\end{enumi}

\smallskip

\noindent
Then there exists a finitely generated group $\Ga$ and a generating set
$\<S\> = \Ga$,
with growth function $h(n) = \ga_{\Ga}^S(n)$, such that:
\begin{enum1}
\item $h(n) < f_1(n)$ for all $n\in \nn$ large enough,
\item $h(n) > f_2(n)$ for infinitely many $n \in \nn$,
\item $h(n) < g_1(n)$ for infinitely many $n \in \nn$,
\item $h(n) \geq  g_2(n)$ for all $n \in \nn$\ts.
\end{enum1}
\end{thm}
\smallskip

Although the conditions are technical, they are mostly mild
in a sense that many natural functions satisfy them.  For example,
condition~(\ref{t:main:con:square}) may seem strong, but notice
that our functions are greater than $\exp(n^\al)$, in which case
$(f_2)^3 \sml f_2$.   Similarly, condition~(\ref{t:main:con:restriction})
may seem restrictive, but in fact, due to Erschler's theorem~\cite{Ers2},
the growth of such Grigorchuk groups can be as large as desired,
even if we do not know anything else about these growth functions,
and with the currently available tools cannot yet control their growth.

\medskip

\subsection{Examples}\label{ss:main-ex}
The condition~(\ref{t:main:con:growth}) implies that the growth of $f_1$
is somewhat faster than that of~$g_1$.  However if the growth
if $g_1$ is close to exponential this condition also implies
that $f_1$ is significantly larger then $f_1$, since $f_1^*(\cdot)$
is small in that case.  To clarify this condition, we list some
examples below.

\begin{ex}\label{ex:main-growth1}
Let $\log g_1(n) \sim n^\al$ and $\log f_1(n) \sim n^\be$, for some
$0< \al < \be \le 1$.  The condition~(\ref{t:main:con:growth}) in this case
says that $\be > 1/(2-\al) > \al$.  This implies that the interval
$(\al,\be)$ cannot be arbitrary and thus, in particular, the
Main theorem cannot imply Brieussel's theorem~\cite{Bri2} (see below).
\end{ex}

\begin{ex}
Let $g_1(n)  \sim n/\log^\al n$, $\log f_2(n) \sim n/\log^\nu n$,
and $\log f_1(n) \sim n/\log^\be n$, where \ts $0< \be \le \nu \le \al$.
The condition~(\ref{t:main:con:growth}) in this case says that $\nu > 1/(\al/\be-2)$.
For example, $\al=5$, $\nu = 3$ and $\be=2$ works, but in
order to have $\nu =\be=2$, one needs $\al>5$.
\end{ex}

\begin{ex}
Let $g_1(n)  \sim n/\log^\al n$ be as before, but now
$\log f_2(n) \sim n/(\log\log n)^\nu$,
and $\log f_1(n) \sim n/(\log\log n)^\nu$, where \ts $0< \be \le \nu$
and $\al>0$.
The condition~(\ref{t:main:con:growth}) in this case says that $\nu > 1/(\al/\be-1)$.
For example, $\al=4$, $\nu = \be=2$ works fine.  More generally,
any $\nu = \be >\al$ satisfy the condition, as
well as $0< \nu = \be < \al-1$.
\end{ex}

\medskip

\subsection{A sketch of the group construction and the proof}
\label{ss:main-sketch}
The group $\Ga$ is constructed from on a sequence of integers $\{m_i\}$ and a
sequences of finite groups $\{H_i \subset \Sym(k_i)\}$ generated by $4$ involutions.
The group $\Ga$ acts on a \emph{decorated binary tree}~$\wh \rT_2$ obtained from
the (usual) infinite binary tree~$\rT_2$ as follows.\footnote{The
description of group~$\Ga$ is written in the language of~\cite{GP},
which is different from the rest of this paper.}  To each vertex on
level~$m_i$ we attach $k_i$ leaves, which are permuted by the group~$H_i$.
The group $\Ga$, like the Grigorchuk group $\bG_\om$, is generated
by~$4$ involutions, whose faithful action is recursively defined.
The definition is similar to the usual one; however, once we reach
a vertex with leaves, one need to specify the action on these leaves,
which is given by a generator of the group $H_i$.  We illustrate the
action of~$\Ga$ on~$\wh \rT_2$ in Figure~\ref{f:decor}, where the set of
leaves of size $H_i$ is decorating all vertices on $m_i$-th level.

\smallskip

\begin{figure}[hbt]
\begin{center}
\psfrag{T}{\small $\rT_2$}
\psfrag{TT}{\small $\wh \rT_2$}
\psfrag{z}{\small $z$}
\psfrag{zz}{\small $z'$}
\psfrag{H}{\hskip-.06cm\small $H_i$}
\epsfig{file=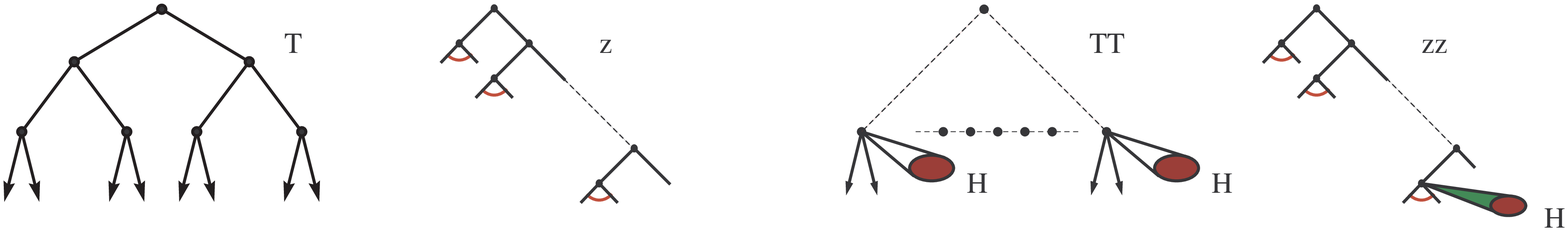,width=14.3cm}
\end{center}
\caption{\small Binary tree~$\rT_2$, generator $z\in \{a,b,c,d\}$ of~$\bG_\om$
acting by transposing selected branches in~$\rT_2$, the decorated binary
tree~$\wh \rT_2$ and generator $z'$ of~$\Ga$ acting on $\wh \rT_2$.}
\label{f:decor}
\end{figure}

Now, the reasoning behind the proof of the Main Theorem is the following.
Roughly speaking, the balls of radius $n$ does not see the group $H_i$
for $m_i >  \log n$. Therefore if the levels~$m_i$ grows sufficiently fast,
the growth of $\Ga$ is similar to the growth of $\bG_\om$ in the last
final interval before the $i$-th level is reached.  However, once we
reach an element in~$H_i$, the growth of $\Ga$ is determined by the
growth of~$H_i$, and can be much more rapid in this period.  If the
groups $H_i$ have logarithmic diameter and their sizes increase sufficiently
fast, we can ensure that the growth of $\Ga$ is as close to exponential
function as desired.

\smallskip

\section{Basic definitions and notations}
\label{s:def}

\subsection{Growth of groups}
Let $f, g: \nn \to \nn$ be two integer functions, such that
$f(n)$, $g(n) \to \infty$ as $n\to \infty$.  We write
\begin{align*}
f \ts \ll \ts g  \quad &  \text{if }  f(n) <  g(c\ts n)\ts, &
\ \ \text{for some} \ \, c >0 &\ \ \text{and infinitely many}
\ \ n\in \nn\ts,
\\
f \ts \sml \ts g \quad & \text{if }   f(n)  <  g(c\ts n)\ts, &
\ \ \text{for some} \ \,  c >0 &\ \ \text{and all sufficiently large} \ \ n\in \nn\ts.
\end{align*}
For example, $n^{100} \. \sml \, 3^n \. \sml \, 2^n$ and
$2^n \ll n^{n\ts (n\nts\nts\nts\mod\nts 2)} \ll n^2$.
Note here that ``$\ll$'' is not transitive since $2^n$ is
\emph{not} \ts $\ll n^2$.  We write $f \sim g$, if \ts
$f \. \sml\. g$ \ts and \ts $g \. \sml\. f$.

Let $\Ga$ be a finitely generated group and $S = S^{-1}$ a
symmetric generating set, $\Ga = \<S\>$.  Denote by
$B_{\Ga,S}(n)$  the set of elements $g\in \Ga$ such that
$\ts \length_{S}(g) \le n$, where $\ts \length_S$ is the
\emph{word length}, and let $\ga_{\Ga}^S(n) = |B_{\Ga,S}(n)|$.
Since for every other symmetric generating set $\<S'\> = \Ga$, we have
$C_1 \length_S(g) \le \length_{S'}(g) \le C_2 \length_S (g)$,
which implies that \ts
$\ga_{\Ga}^S(n) \ts\sml \ts \ga_{\Ga}^{S'}(n) \ts \sml\ts \ga_{\Ga}^S(n)$.
In other words, the asymptotics of $\ga^S_{\Ga}(n)$ are independent
of the generating set $S$, so whenever possible we will write
$\ga_\Ga(n)$ for simplicity.

Group $\Ga$ has \emph{exponential growth} if
$\ga_{\Ga}(n) \ts \grt \ts\exp(n)$, and \emph{polynomial growth} if
$\ga_{\Ga}(n) \ts \sml \ts n^c$ for some $c>0$.  Similarly, $\Ga$ has
\emph{intermediate growth} if $\ga_{\Ga}(n) \grt n^c$ for all $c>0$, and
$\ga_{\Ga}(n) \sml \exp f(n)$ for some $f(n)/n \to 0$,
as $n\to \infty$.

{\small

\begin{rem}
Using this notation, the Main Theorem says that for any functions
$$
f_1\, \grt \, f_2 \, \grt \, g_1 \, \grt \, g_2 \. = \. \ga^S_{\bG_\om}\ts,
$$
which satisfy additional technical assumptions, and
where $\bG_\om$ is a \emph{Grigorchuk group} of intermediate growth,
there exists a group~$\Ga$, whose growth function satisfies
$$
f_1 \,\. \grt \,\. \ga_\Ga \.\, \grt \.\, g_2\ts,
\quad
\ga_\Ga \ts \gg \ts f_2\ts, \quad \ \text{and}
\quad
\ga_\Ga \ts \ll \ts g_1\ts.
$$
In the special case when $f_1 \sim f_2$ and $g_1 \sim g_2$, this means that
$$
f_1 \,\. \grt \,\. \ga_\Ga \,\. \grt \,\. g_2 \,
\quad \mbox{and} \quad
f_1 \ts \ll \ts \ga_\Ga \ts \ll \ts g_2\ts.
$$
\end{rem}
}

\smallskip

For a group of intermediate growth, define
$$
\al(\Ga) \. = \. \lim_{n\to \infty} \frac{\log \log \ga_{\Ga}(n)}{\log n}
\ \ \text{if this limit exists},
$$
$$
\al_+(\Ga) \. = \. \limsup_{n\to \infty} \. \frac{\log \log \ga_{\Ga}(n)}{\log n}
\quad
\mbox{and}
\quad
\al_-(\Ga) \. = \. \liminf_{n\to \infty} \. \frac{\log \log \ga_{\Ga}(n)}{\log n}\..
$$

Our Main theorem (and the result in~\cite{Bri2}), shows that $\al(\Ga)$ is
does not necessarily exist, the first construction of this kind.
The result of Brieussel mentioned in
the introduction, implies that there exist a group~$\Ga$ of intermediate
growth with $\al(\Ga) =\nu$, for any given $\al \le \nu \le 1$.

\smallskip

In addition to the growth function~$\ga_G^S$, we define a
\emph{normal growth function} \ts $\wt \ga_{G,X}^S$, as the
number of elements in the group~$G$ which can be expressed
as words in the free group on length~$n$, which also
lie in the normal closure of elements $X\ssu G$.  In this
paper we consider only the case $X=\{r\}$, which we denote
$\wt \ga_{G,r}^S$.

\subsection{Notation for groups and their products}
To simplify the notation, we use
$\zz_m$ for $\zz/m\zz$ and trust this will not leave to any
confusion.  Let $\PSL_2(N)$ denotes group $\PSL_2(\Z/N\Z)$,
$\F_k$ denotes the free group on~$k$ generators.  By
$G \toto H$ we denote an \emph{epimorphism} between
the groups.

The group with presentation
$$
\cG = \la a,b,c,d \ts \mid \ts a^2=b^2=c^2=d^2=bcd=1 \ra
$$
will play a central role throughout the paper, as all our groups
and also all Grigorchuk groups are homomorphic images of~$\cG$.
We call it a \emph{free Grigorchuk group}.

The \emph{direct product} of groups $G$ and $H$ is denoted $G\oplus H$,
rather than more standard $G \times H$.  This notation
allows us to write infinite product as $\ts \bigoplus \ts G_i$,
where all but finitely many terms are trivial
and we will typically omit the index of summation.
We denote by \ts $\prod G_i$ \ts the (usually uncountable) group of
sequences of group elements, without any finiteness conditions.
Of course, groups \ts
$\bigoplus \ts G_i$ \ts nor \ts $\prod G_i$ \ts are not
finitely generated.

Finally, let $H \wr G = G\ltimes H^\ell$ denotes the permutation
wreath product of the groups, where $G \ssu \Si_\ell$
is a permutation group.


\subsection{Marked Groups and their homomorphisms}
All groups we will consider will have ordered finite generating
sets of the same size~$k$.  Whenever we talk mention a group~$G$,
we will mean a pair $(G,S)$ where $S=\{s_1,\ldots,s_k\}$ is a ordered
generating set of $G$ of size $k$.  Although typically
described with Cayley graph, the order on the generators is
crucial for our results.  We call these \emph{marked groups},
and $k$ will always denote the size of the generating set.
By a slight abuse of notation, will often drop~$S$ and refer
to a \emph{marked group~$G$}, when $S$ is either clear from
the context or not very relevant.

Throughout the paper, the homomorphisms between marked
groups will send one generating set to the other.  Formally,
let $(G,S)$ and $(G',S')$ be marked groups, where
$S=\{s_1,\ldots,s_k\}$ and $S'=\{s_1',\ldots,s_k'\}$.  Then
$\phi: (G,S) \to (G',S')$ is a \emph{marked group homomorphism}
if $\phi(s_j) = s_j'$, and this map on generators extends to
the (usual) homomorphism between groups: $\phi: G\to G'$.

An equivalent way to think of marked groups is to consider
epimorphisms $\F_k \toto G_1$,  $\F_k \toto G_2$, so that
the map between groups correspond to commutative diagrams
\begin{diagram}[nohug,height=7mm]
     &              &  G_1      \\
\F_k & \ruOnto(2,1) & \dTo   \\
     & \rdOnto(2,1) & G_2     \\
\end{diagram}

\subsection{Direct sums of marked groups}

Let $\{G_i\}$ be a sequence of marked groups defined above,
or more formally $\{(G_i,S_i)\}$.
Denote by $G=\bp G_i$ the subgroup of $\prod G_i$  generated
by diagonally embedding the generating sets~$S_i$, see Definition~\ref{d:product_of_marked_groups}.
Of course, group~$G$ critically depends on the ordering
of elements in~$S_i$.

\subsection{Miscellanea}
With $\om = (x_1,x_2,\dots)$ will denote an infinite word in
$\{0,1,2\}$ which will be used to construct the Grigorchuk group
$\bG_\om$. Such word is call stabilizing if all $x_i$ are eventually
the same. In this case the group $\bG_\om$ becomes virtually nilpotent.

We use $\log n$ to denote natural logarithms, but normally
the base will be irrelevant.  The radius of balls in the
groups will be denoted with~$n$.
Finally, we use $\nn=\{1,2,\ldots\}$.

\smallskip

\section{Limits of Groups}
\label{sec_limits}

All marked groups we consider will have ordered finite generating
sets of the same size~$k$, and all maps between marked groups will
send one generating set to the other.

\begin{defn}
\label{d:product_of_marked_groups}
Let $\{(G_i,S_i)\}, i\in I$ be a sequence of marked groups with generating sets
$S_i=\{s_{i1},\dots,s_{ik}\}$.
Define the $(\Gamma,S) = (\bp G_i,S)$ to be the subgroup of $\prod G_i$ generated by
diagonally embedding the generating sets of each $G_i$, i.e,
$\bp G_i = \< s_1,\dots, s_k\>$ where $s_j=\{s_{ij}\} \in \prod G_i$.
Notice that $\Gamma$ comes with canonical epimorphisms
$\zeta_i: \Gamma \toto G_i$. Often the generating sets will be
clear from the context and will simply use $\Ga = \bp G_i$.
When the index set contains only $2$ elements we denote the product
by $G_1\op G_2$.
\end{defn}

{\small
\begin{remark}
The notations $\bp G_i$ and $G_1 \op G_2$ are slightly misleading since
these products depend not only on the groups but also on the generating sets.
In this paper all groups are marked and come with a fixed generations set,
which justifies this abuse of the notation.
\end{remark}

\begin{remark}
\label{re:universal}
The group $G_1\op G_2$ satisfies the following universal property --
for any marked group $H$ such that the left 
two triangles commute,
there exits a homomorphism $H \to G_1\op G_2$.
\begin{diagram}[nohug,height=13mm]
     &              &   &  G_1     &                                      \\
\F_k & \rOnto       & H \ruOnto(3,1)\ruOnto(1,1)
                        & \rDotsto & G_1 \op G_2 \luOnto(1,1)^{\zeta_1}   \\
     & \rdOnto(3,1) &   &  G_2 \rdOnto(1,1) \ldOnto(1,1)_{\zeta_2}
                                   &                                      \\
\end{diagram}
\end{remark}
}

\smallskip

\begin{lemma}
\label{l:balls_in_products}
\label{lower_bound_for_ball}
\begin{enumi}
\item
If $G_i$ is any sequence of marked groups then growth function of
$\Gamma=\bp G_i$ is larger than the growth functions of each $G_i$, i.e.,
$$
\ga_{\Gamma}(n) \geq \ga_{G_i}(n) \quad \mbox{for all }~i\ts.
$$
\item
If $G_1$ and $G_2$ are two marked groups then the growth function of
$G_1\op G_2$ is bounded by the product of the growth functions for $G_i$
$$
\ga_{G_1\op G_2}(n) \leq \ga_{G_1}(n) \cdot \ga_{G_2}(n)\ts.
$$
\end{enumi}
\end{lemma}

\begin{proof}
By definition of the product of marked groups the map $\zeta_i: \Gamma\to G_i$
is not only surjective, but also satisfies
$\zeta_i(B_{\Gamma,S}(n)) = B_{G_i,S_i}(n)$, which implies the first part.
The injectivity of the product of the projections $\zeta_1$ and $\zeta_2$ \
and the observation
$$
\zeta_1\times \zeta_2 : B_{\Gamma,S}(n) \hookrightarrow B_{G_1,S_1}(n) \times B_{G_2,S_2}(n),
$$
imply the second part.
\end{proof}

\begin{defn}
We say that the sequence of marked groups $\{(G_i,S_i)\}$ \emph{converge}
(in the  the so-called \emph{Chabauty topology})
to a group $(G,S)$ if for any $n$ there exists $m=m(n)$ such that
such that for any $i>m$ the ball of radius $n$ in $G_i$
is the same as the ball of radius $n$ in~$G$.
We write $\lim G_i = G$.
\end{defn}

Equivalently, this can be stated as follows:
if
$R_i = \ker( \F_k \toto G_i)$ and $R = \ker( \F_k \toto G)$
then
$$
\lim_{i\to \infty} R_i \cap B_{\F_k}(n) \. = \. R \cap B_{\F_k}(n)\ts,
$$
i.e., for a fixed $n$ and  sufficiently large $i$ the sets
$R_i \cap B_{\F_k}(n)$ and $R \cap B_{\F_k}(n)$ coincide.

\begin{lemma}\label{l:product}
Let $\{G_i\}$ be a sequence of marked groups which converge to a marked group~$G$.
Define the $\Gamma = \bp G_i$, then there is an epimorphism $\pi: \Gamma \toto G$.
Moreover, the kernel of~$\pi$ is equal to the
intersection \ts $\Gamma \cap \bigoplus G_i$\ts.
\end{lemma}
\begin{proof}
There is an obvious map $\pi$ which sends the generators of $\Gamma$ to
the generators of~$G$.
A word $w$ represents the trivial element in $\Gamma$ if and only if $w$ is
trivial in all $G_i$. Therefore this word is trivial in infinitely many of $G_i$
and is as well trivial in the limit $G$, i.e.,
the map $\pi$ extends to a group homomorphism.

The convergence of $\{G_i\} \to G$ implies that
if a word $w\in \F$ of length $n$ which
evaluates to $\{g_i\} \in \prod G_i$
is in the kernel of $\pi$ then the components $g_i$ have to be
trivial for large $i$ (otherwise the ball of radius $n$ in the Cayley graph
of $G_i$ will  be different form the one in~$G$).
Therefore, the word $w \in \bigoplus G_i$.  The other
inclusion is obvious.
\end{proof}

Lemma~\ref{l:product} allows us to think of $G$ as the \emph{group at infinity}
for~$\Gamma$.
We will be interested in sequences of groups which satisfy the additional
property that
\begin{equation*}
\label{condition}
\tag{spliting}
\lim_{i\to \infty} G_i = G
\quad \mbox{and} \quad
\Gamma \, = \, \left[\bp G_i\right] \cap \left[\bigoplus G_i\right]
\, = \, \bigoplus N_i\ts,
\end{equation*}
where $N_i$ are normal subgroups of $G_i$.

\begin{lemma}
If the groups $G_i$ satisfy the condition~(\ref{condition}), then
there exists a group homomorphism $\pi_i:G \to G_i/N_i$ which makes
the following diagram commute: 
\begin{diagram}[height=7mm]
1 & \rTo & \ker \pi & \rInto & \Gamma & \rOnto{\pi} & G  & \rTo & 1\\
  &   & \dOnto{\zeta_i}  &   & \dOnto{\zeta_i}  &  &\dDotsto{\pi_i} & &\\
1 & \rTo & N_i & \rInto & G_i & \rOnto & G_i/N_i  & \rTo & 1
\end{diagram}
\end{lemma}
\begin{proof}
The existence and uniqueness of the homomorphism $\pi$ follows from the
exactness of the rows in the diagram above.
\end{proof}

Using these maps one can obtain estimates for the size of the ball
in the groups~$\Gamma$:

\begin{lemma}
\label{upper_bound_for_ball}
If the groups $\{G_i\}$ satisfy the condition~(\ref{condition}) and
$B_{G_i}(n)$ is the same as $B_G(n)$
for $i > m$ then
$$
\ga_\Gamma(n) \. \leq \. \ga_{G}(n) \, \prod_{j \leq m} \. |N_j|\ts.
$$
\end{lemma}
\begin{proof}
If two elements  $g,h \in B_\Gamma(n)$ are send to the same element
elements in $B_G(n)$ then are also the same in $B_{G_j}(n)$ for all
$j> m$, i.e., their difference $g^{-1} h$ is inside
$$
\Gamma \cap \bigoplus_{j\leq m} G_i = \bigoplus_{j \leq m} N_j.
$$
Therefore, the fibers of the restriction of $\pi$ to $B_\Gamma(n)$
have size at most
$\prod_{j \leq m} \ts |N_j|$,
which implies the inequality in the lemma.
\end{proof}

\smallskip

\section{The Grigorchuk group}\label{s:gri}

\subsection{Basic results}
In this section we present variations standard results on the Grigorchuk
groups $\bG_\om$ (cf.~Subsection~\ref{ss:fin-seq}).
Rather than give standard definitions as a subgroup of $\Aut (\rT_2)$,
we define $\bG$ via its properties.  We refer to~\cite{GP,Har1}
for a more traditional introduction and most results in this
subsection.

\begin{defn}
\label{d:twist}
Let $\varphi: \cG \toto \cG$ denote the automorphism of order $3$ of the group
$G$ which cyclicly permutes the generators $b$, $c$ and $d$, i.e.,
$$
\varphi(a) = a,
\quad
\varphi(b) = c,
\quad
\varphi(c) = d,
\quad
\varphi(d) = b.
$$
\end{defn}

\begin{defn}
Let $\pi: \cG \toto H$ be an epimorphism, i.e., suppose group $H$
comes with generating set consisting of $4$ involutions
$\{a,b,c,d\}$ which satisfy $bcd=1$.  By $F(H)$
we define the subgroups of $H \wr \zzz  =  \zzz \ltimes (H \oplus H)$
generated by the elements $A,B,C,D$ defined as
$$
A = (\xi; \ts 1,1)\ts,
\quad
B = (1; \ts a,b)\ts,
\quad
C = (1; \ts a,c)
\quad \text{and} \ \
D = (1; \ts 1, d)\ts,
$$
where $\xi^2=1$ is the generator of~$\zzz$.
It is easy to verify that $A,B,C,D$ are involutions which satisfy $BCD=1$,
which allows us to define an epimorphism $\wt F(\pi): \cG \to F(H)$.

The construction can be twisted by the powers automorphism $\varphi$
$$
\wt F_x(\pi): = \wt F(\pi \circ \varphi^{-x}) \circ \varphi^{x}.
$$
\begin{diagram}[nohug,height=7mm]
\cG & \rTo{\varphi^x} & \cG          & &
\cG & \rTo{\varphi^x} & \cG          \\
&\rdOnto(1,2)_{\pi_x} \ldOnto(1,2)_{\pi} &    & &
&\rdOnto(1,2)_{\wt F(\pi_x)} \ldOnto(1,2)_{\wt F_x(\pi)} &    \\
    &       H         &              & &
    &       F_x(H)    &              \\
\end{diagram}

An equivalent way of defining the group $F_x(H)$ is as the subgroups
generated by
\begin{align*}
A_0 =&  (\xi; \ts 1,1)\ts,  \quad &
B_0 =& (1; \ts a,b)\ts, \quad &
C_0 =& (1; \ts a,c) \quad & 
D_0 =& (1; \ts 1, d)\ts, \\
A_1 =&  (\xi; \ts 1,1)\ts,  \quad &
B_1 =& (1; \ts a,b)\ts, \quad &
C_1 =& (1; \ts 1,c) \quad & 
D_1 =& (1; \ts a, d)\ts, \\
A_2 =&  (\xi; \ts 1,1)\ts,  \quad &
B_2 =& (1; \ts 1,b)\ts, \quad &
C_2 =& (1; \ts a,c) \quad & 
D_2 =& (1; \ts a, d)\ts,
\end{align*}
\end{defn}

{\small
\begin{remark}
Strictly speaking, the notation $F_i(H)$ is not precise since in
order to define this group we need to specify a generating set,
thus the correct notation should be $\wt F_i(\pi)$.
However since all groups~$H$ are marked, i.e., come with an epimorphism
$\cG \toto H$, this allows us to slightly simplify the notation.
\end{remark}
}

\begin{prop}
Each $F_x$ is a functor form the category of homomorphic images of $\cG$ to itself,
i.e., a group homomorphism $H_1 \to H_2$ which preserves the generators induces,
a group homomorphism $F_x(H_1) \to F_x(H_2)$.
\begin{diagram}[nohug,height=8mm,width=13mm]
    &                    &  H_1      & \quad\quad &
    &                    &  F_i(H_1) \\
\cG & \ruOnto(2,1)^{\pi_1} & \dOnto_{\theta}& &
\cG & \ruOnto(2,1)^{\wt F_x(\pi_1)} & \dOnto_{F_x(\theta)} \\
    & \rdOnto(2,1)_{\pi_2} & H_2     & &
    & \rdOnto(2,1)_{\wt F_x(\pi_2)} & F_x(H_2)
\end{diagram}
\end{prop}
\begin{prop}
\label{pro:F_preserves_products}
The functors $F_x$ commutes with the products of marked groups, i.e.,
$$
F_x\left(\bp H_j\right) = \bp F_x(H_j).
$$
\end{prop}
\begin{proof}
This is immediate consequence of the functoriality of $F_i$ and the universal
property of the products of marked groups. Equivalently one can check directly
from the definitions.
\end{proof}


\begin{defn}
One can define the functor $F_\om$ for any finite word $\om \in \{0,1,2\}^*$
as follows
$$
F_{x_1x_2\dots x_i}(H):= F_{x_1}(F_{x_2}(\dots F_{x_i}(H)\dots))
$$
If $\omega$ is an infinite word on the letters $\{0,1,2\}$ by
$F_\omega^i$ we will denote the functor $F_{\omega_i}$ where
$\omega_i$ is the prefix of $\omega$ of length $i$.
\end{defn}
\begin{thm} [cf.~\cite{Gri3}]
\label{t:grig}
The Grigorchuk group $\bG$ is the unique group such that
$\mathbf{G} = F_{012}(\mathbf{G})$.
\end{thm}

{\small
\begin{remark}
In~\cite{Gri3}, Grigorchuk defined a group $\bG_\omega$ for any
infinite word $\omega$.  One way to define these groups
is by $\bG_{x\omega} = F_x(\bG_\omega)$, where $x$ is any
letter in $\{0,1,2\}$.  The \emph{first Grigorchuk group}
is denoted $\bG = \bG_{(012)^\infty}$, which
corresponds to a periodic infinite word.
If the word $\om$ stabilize then the group $\bG_\om$
is virtually nilpotent and has polynomial growth.
\end{remark}
}

Although we will not use Theorem~\ref{t:grig}, the following
constructions gives the idea of the connection.
Let $\bG_{\omega,i} = F^i_\omega(\mathbf{1})$, where $\mathbf{1}$ denotes the
trivial group with one element (with the trivial map
$\cG \toto \mathbf{1}$).

\begin{prop}\label{p:grig-rooted}
There is a canonical epimorphism $\bG_\omega \toto \bG_{\omega,i}$.
The groups $\bG_{\omega,i}$ naturally act on finite binary rooted tree
of depth $i$ and this action comes from the standard action
of the Grigorchuk group on the infinite binary tree~$\rT_2$. \qed
\end{prop}

{\small
\begin{remark}
\label{r:inclusion_into_wreath_product}
The group $F^i_\om(H)$ is a subgroup of the permutational wreath product
$H \wr_{X_i} \bG_{\om,i}$, where $X_i$ is the set of leaves of the
binary tree of depth~$i$ (cf.~Subsection~\ref{ss:main-sketch}).
\end{remark}
}

\smallskip

\subsection{Contraction in Grigorchuk groups}

\begin{lemma}
\label{l:contracting_property_of_G}
Let $\pi:\cG \toto H$ be an epimorphism, i.e., group~$H$ is
generated by $4$ \emph{nontrivial} involutions which satisfy $bcd=1$.
If the word $\om$ does not stabilize,
then the balls of radius $n \leq \cff(m)$ in the groups
$F^m_{\om}(H)$ and $\bG_{\om}$
coincide, where $\cff(m) = 2^{m}-1$  is strictly
increasing function $\cff: \nn\to\nn$.
\end{lemma}

\begin{proof}
It is enough to show that the set of words of length $2\cff(m)$
which are trivial in $F^m_\om(H)$ is the same as the ones which are
trivial in $\bG_\om$.

Observe that every word $w\in \cG$ can be expanded to $(\xi^{a_w}; w',w'') \in F(\bar H)$
where $w'$ and $w''$ are words of length $\leq (|w|+1) /2$. If $a_w \not =0$
then $w$ is not zero in $F(\bar{H})$, for any group $\bar{H}$.
Iterating this $m$ times, shows that any word $w$ of length $ < 2\cff(m)$
is either nontrivial in both $F_\om^m(H)$ and $\bG_\om$; or evaluates to
many words of length at most $1$ acting on the copies of $H$.
If one of these words in nontrivial then $w$ otherwise it is trivial.

Here we are using that the non-stabilization on $\om$ implies that
the elementals $a,b,c,d$ are nontrivial in $\bG_{\bar{\om}}$ for any
suffix $\bar{\om}$ of $\om$.
\end{proof}

{\small
\begin{remark}
One can show that a stronger result holds if $\om$ does not contain
$0^k$, $1^k$ and $2^k$ as subwords.  Indeed, then the balls is
$F^m_{\om}(H)$ and $\bG_{\om}$ of radius $\cff(m)$
are the same as the balls in $\bG_{\om, m+k+1}$. The last group groups is
of the form $F^m_{\om}(H')$ where $H' = F_{\om'}(\mathbf{1})$ where $\om'$
is a subword of $\om$ of length $k+1$ and the conation on $\om$
implies that the generators $a,b,c$ and $d$ are nontrivial in~$H'$.
\end{remark}

\begin{remark}
Here we use that the length of each word $w'$ and $w''$ is shorter than $w$.
In many cases one can also show that the sum of the lengths
(or some suitably defined norm) of these words is less that that of $w$.
Such contracting property is used to obtain upper bounds for the growth
of~$\bG_\om$, see~\cite{Bar1,BGS,Gri3,MP}.
%
\end{remark}
}

We conclude with an immediate corollary of the
Proposition~\ref{p:grig-rooted} and
Lemma~\ref{l:contracting_property_of_G},
which can also be found in~\cite{Gri5}.

\begin{cor}
Let $\{\cG \toto H_i\}$  be any sequence of 
groups generated
by $k=4$ nontrivial involutions and let $\{m_i\}$
be an increasing sequence. Then the sequence of groups $\{F^{m_i}_\omega(H_i)\}$
converge (in the Chabauty topology) to~$\bG_\omega$.
\end{cor}

{\small
\begin{remark}
This can be used as an alternative definition of the groups $\bG_\omega$,
which shows that there exists a canonical epimorphism $\cG \toto \bG_\omega$.
\end{remark}
}

\smallskip

\subsection{Growth lemmas}
Let $r$ denote the element
\ts $[c,[d,[b,(ad)^4]]] \in \cG$ \ts  and let $r_x = \varphi^x(r)$
be its twists by the automorphism $\varphi$ described in Definition~\ref{d:twist}.

\begin{lemma}
\label{l:diam_of_F^i}
Let $\cG\toto H$ be a finite image of $\cG$ which
normally generated by element $r_{x_{k+1}}$ defined above.
Then the kernel of the map
$F^k_\omega(H) \toto  \bG_{\omega,k}$  induced by $F^k_\omega$ from the trivial
homomorphism $H\toto\mathbf{1}$, is isomorphic to $H^{\oplus 2^{k}}$.
Moreover, there exists a word $\eta_{\omega,k} \in F$ of length
$\le \rK\cdot 2^k$,
such that such the image of $\eta_{\omega,k}$ in $F^k_\omega(H)$
normally generated this kernel and $\eta_{\omega,k}$ is trivial in
$F^{k+1}_\om(H')$, for every $\cG\toto H'$.
\end{lemma}
\begin{proof}
Consider the substitutions $\sigma,\tau$ (endomorphisms $\cG \to \cG$),
defined as follows:
\begin{itemize}
\item $\sigma(a) = aca$  and $\sigma(s) = s$, for $s\in \{b,c,d\}$\ts,
\item $\tau(a) = c$,  $\tau(b) = \tau(c) = a$ and $\tau(d) = 1$\ts.
\end{itemize}
It is easy to see that for any word $\eta$, the evaluation of $\sigma(\eta)$
in $F(H)$ is equal to
$$
\bigl(1;\tau(\eta),\eta\bigr) \, \in \, \{1\} \times H \times H \, \subset \, H \wr \zzz \ts.
$$

Define words $\{w_i\}$ for $i=0,\dots, k$ as follows:
$w_0 = r_{x_{k+1}}$ and $w_{i+1} = \sigma_{x_{k-i}}(w_{i})$
where $\sigma_{x_i} = \varphi^{x_i} \sigma \varphi^{-x_i}$
the the twist of the substitution $\sigma$.
Notice that all these words have the form $[c,[d,[b,*]]]$ because
$\sigma_{x_i}$ fixes $b$, $c$ and $d$. Therefore $\tau_x(w_i) =1$

By construction the word $\eta_{\om,k} = w_k$ evaluates in $F^k_\om(\bar{H})$
to $r_{x_{k+1}}$ in one of the copies of $\bar{H}$, for any group
$\bar{H}$.  The expression $(ad)^4$ inside $r$ ensures that $r_{x_{k+1}}$
is trivial if $\bar{H}$ is of the form $F_{x_{k+1}}(\bar{H}')$,
which proves the last claim.

The first claim follows form the transitivity of the action of $\bG_\om$
(and $F^k_\om(H)$) on the $m$-th level of the binary tree and the
assumption that $H$ is normally generated by $r_{x_{k+1}}$.
\end{proof}

{\small
\begin{remark}
The lemma says that if $H$ is normally generated by the element~$r$,
then the inclusion in Remark~\ref{r:inclusion_into_wreath_product}
is an equality.
\end{remark}
}

\smallskip

\begin{cor}
\label{cor:weak_lower_bounds_for_balls}
For~$H$ as in Lemma~\ref{l:diam_of_F^i} and every integer $n\ge 1$, we have:
$$
\ga_{H}(n) \. \le \. \ga_{F^k_\om(H)}(\ccc_k\ts n) \ts,
\quad \text{where} \ \  \ccc_k=2^{k+1}-1.
$$
\end{cor}

\begin{proof}
Use that $\sigma_{k,\om}(B_{H}(n)) \subset B_{F^k_\om(S)}(\ccc_k\ts n)$
because the composition $\sigma_{k,\om}$ of $k$ substitutions
$\sigma_{x_j}$ increases the lengths of the words at most~$2^{k+1}-1$
times.
\end{proof}


\begin{cor}
\label{cor:lower_bounds_for_balls}
For~$H$ as in Lemma~\ref{l:diam_of_F^i} and every integer
$n\ge 1$ and any $t <  2^k$,
we have:
$$
\wt\ga_{H,r_{x_{k+1}}}(n)^t \. \le \. \ga_{F^k_\om(H)}(\ccc_k\ts t \ts n) \ts,
$$
where $\wt \ga_H(n)$ is the \emph{normal growth function}, i.e.~the number
of elements in $H$ which can be expressed as words of length less then~$n$
in the normal subgroup~$X = \<r_{x_{k+1}}\>^{\F_4}$ of the free
group~{\rm $\ts \F_4=\<a,b,c,d\>$.}
\end{cor}
\begin{proof}
As before, but use the fact that there are many copies of~$H$.
\end{proof}

\smallskip

\section{Growth in $\PSL_2(\zz_N)$}
\label{s:technical}

For the proof of Theorem~\ref{t:main}, we need the following
technical result:

\begin{lemma}
\label{l:gen_set_in_some_groups}
\label{gen_set_in some_groups}
Let $N$ such that $-1$ is a square in $\zz_N$ and $2 \not\! | N$, i.e.,
the only prime factors which appear in the prime decomposition of $N$
are of the form $p=1 \mod 4$.
Then there exist a generating set $S_N=\{a,b,c,d\}$ of the group $H_N=\PSL_2(\zz_N)$
such that
\begin{enum1}
\item
\label{l:gen_set_in_some_groups:con:marked}
there is an epimorphism of marked groups
$\cG \toto H_N=\PSL_2(\zz_N)\ts$,
\item
\label{l:gen_set_in_some_groups:con:simple}
the group $H_N$ is normally generated
by the image of element \ts $r=[c,[d,[b,(ad)^4]]]$\ts,
\item
\label{l:gen_set_in_some_groups:con:exp}
$\ga_{H_N}(n) > \exp(n/\rK)$,
for $n < \rK\log |H_N| < 3 \rK \log N$, and $\rK >0$ is an absolute constant,
\item
\label{l:gen_set_in_some_groups:con:strong_exp}
$\wt \ga_{H_N,r}(n) > \exp(n/\rK)$,
for $n < \rK'\log |H_N| < 3 \rK' \log N$, and $\rK' >0$ is an absolute constant.
\end{enum1}
\end{lemma}


Here property~(\ref{l:gen_set_in_some_groups:con:exp})
means that the size of balls in the Cayley graphs
of $\PSL_2(\zz_N)$ grow exponentially.
For the proof of Corollary~\ref{cor:osc-growth}
we  do not really need the exact form of these groups nor property~(\ref{l:gen_set_in_some_groups:con:strong_exp}),
only the fact the their sizes go to infinity.
However the proof of Theorem~\ref{t:main} uses that these groups
are related to $\PSL_2(\zz)$.
\begin{proof}
Consider the following matrices in $\PSL_2\bigl(\zz[\ii,1/2]\bigr)$,  where $\ii^2=-1$,
$$
a = \left(\begin{array}{cc} \ii & \ii/4 \\ 0   & -\ii  \end{array}\right),
\quad \
b = \left(\begin{array}{cc}   0 & \ii   \\ \ii & 0   \end{array}\right),
\quad \
c = \left(\begin{array}{cc}   0 & 1     \\  -1 & 0    \end{array}\right),
\quad \
d = \left(\begin{array}{cc} \ii & 0     \\  0  & -\ii \end{array}\right).
$$
A direct computation shows that these elements are of order~$2$ and
$bcd=1$, i.e., there is a (non-surjective%
\footnote{The images contains $\PSL_2\bigl(\zz[1/2]\bigr)$
as a subgroup of index $2$.}%
) homomorphism
$\cG \to \PSL_2\bigl(\zz[\ii,1/2]\bigr)$.
Moreover, we have
$$
(ad)^4 = \left(\begin{array}{cc}1 & -1 \\ 0 & 1\end{array}\right),
\qquad
[c,[d,[b,(ad)^4]]] = \left(\begin{array}{cc} -1 & 2 \\ 2 & - 5\end{array}\right).
$$
This implies that the image of $\{a,b,c,d\}$ in $\PSL_2(\zz_N)$
satisfies properties~(\ref{l:gen_set_in_some_groups:con:marked})
and~(\ref{l:gen_set_in_some_groups:con:simple}),
because~$r$ is not contained in any proper finite index normal subgroup of the image.

Property~(\ref{l:gen_set_in_some_groups:con:exp})
is satisfied because the standard expander generators of $\PSL_2(\zz_N)$
can be expressed as short words in the generators (see e.g.~\cite{HLW,Lub-book})\.:
$$
\left(\begin{array}{cc}1 & 1 \\ 0 & 1\end{array}\right) = \ts (da)^4
\quad \ \text{and} \ \ \,
\left(\begin{array}{cc}1 & 0 \\ 1 & 1\end{array}\right) = \ts c\ts (ad)^4c\ts.
$$

Property~(\ref{l:gen_set_in_some_groups:con:strong_exp})
follows since $r$ and $r^a$ generate
a non-solvable subgroup of $\PSL_2(\zz)$.  Now the recent expansion
results~\cite{BG1}, imply that the Cayley graphs of
$\PSL_2(\zz_N)$ with respect to the images of $r$ and $r^a$
are expanders.
This completes the proof (cf.~Subsection~\ref{ss:fin-exp}).
\end{proof}

{\small
\begin{remark}
Twisting by the automorphism $\varphi^x$ one sees that the lemma remains valid
if $r$ is replaced by~$r_x$.
\end{remark}

\begin{remark}
\label{r:product}
If $\{p_j\}$ is a finite sequence of different primes which are
$3\mod 4$ then the product
of $\PSL_2(p_j)$ as marked groups (with resect to the generating sets constructed above) is
$$
\bp \PSL_2(p_j) = \PSL_2\left(\zz_N\right)
\quad \mbox{where }
N= \prod p_i.
$$
\end{remark}
}

The following is an immediate consequence of
Corollaries~\ref{cor:weak_lower_bounds_for_balls} and~\ref{cor:lower_bounds_for_balls}:

\begin{cor}
\label{cor:weka_size_of_balls}
\label{weak_size_of_balls}
Let $N$ satisfies the conditions of Lemma~\ref{gen_set_in some_groups}.
The size of a ball of radius $n < \ddd_i \log |\PSL_2(\zz_N)|$ inside
$F^i_\om(\PSL_2(\zz_N))$ is more than $e^{n/\ddd_i}$ where $\ddd_i = 2^{i+1}\ts\rK$.
\end{cor}

\begin{cor}
\label{cor:size_of_balls}
\label{size_of_balls}
Let $N$ satisfies the conditions of Lemma~\ref{gen_set_in some_groups}.
The size of the intersection of a  ball of radius
$n < 2^i\ddd_i' \log |\PSL_2(\zz_N)|$ inside $F^i_\om(\PSL_2(\zz_N))$, with the subgroup
$\PSL_2(\zz_N)^{2^i}= \ker\{F^i_\om(\PSL_2(\zz_N)) \to \bG_{\om,i}\}$
is more than $e^{n/\ddd_i'}$ where $\ddd_i' = 2^{i+1}\ts\rK'$.
\end{cor}

\smallskip

\section{Proof of the Oscillating Growth Theorem }
\label{s:proof}

We are going to prove the following result, which implies
the Oscillating Growth Theorem (see Corollary~\ref{cor:osc-growth}),
and is a stepping stone to the proof of the Main Theorem.

\begin{thm}
\label{t:general}
For every  admissible integer functions $f,g: \nn\to \nn$, such that
$\lim g(n)/\ga_{\bG_\omega}(n) = \infty$
and any sequences of integers $\{a_i\} \to \infty$ and $\{b_i\} \to \infty$,
there exists a finitely generated group $\Ga$ and a
generating set $\<S\> = \Ga$, such that
$$
\ga_{\Ga}^S(n) \.  < g(n) \ \ \text{for infinitely many} \ \ n\in \{b_i\}\ts,
$$
and
$$
\ga_{\Ga}^S(n) \. > f(n) \ \ \text{for infinitely many} \ \ n\in \{a_i\}\ts.
$$
\end{thm}

{\small
\begin{remark}
The reason for including the subsequences $a_i$ and $b_i$ is to be able to ensure that for any $f \gg f'$ and $g \ll g'$ then we have
$\ga_\Ga \gg f'$ and $\ga_\Ga \ll g'$.
In particular we can guarantee that
$$
\liminf_{n\to \infty} \ts \log_n\log \ga_\Gamma(n) \. = \. \liminf_{n\to \infty} \ts \log_n\log \ga_{\bG_\om}(n)\ts.
$$
\end{remark}
}

\begin{proof}[Proof of Theorem~\ref{t:general}]
The group $\Gamma$ will be the product of marked groups
$G_i = F^{m_i}_\omega(H_i)$ where $H_i=\PSL_2(p_i)$ and $\{m_i\}$ and $\{p_i\}$
are sequences which grow sufficiently fast constructed using the
functions $f$ and $g$. Lemmas~\ref{l:contracting_property_of_G}
and~\ref{l:diam_of_F^i} imply that the sequence of groups $G_i$
converge to $\bG$ and satisfies the
condition~(\ref{condition}) with $N_i = H_i^{\oplus 2^{m_i}}$.

By Corollary~\ref{size_of_balls}, the growth of $\ga_\Ga$ is faster that
the each $\ga_{G_i}$.  When $p_i$ is sufficiently large one can find
$n_i \in \{a_j\}$ such that
\begin{equation*}
\label{eq:lowerbound}
\tag{lower}
\ga_\Ga (n_i) \geq \ga_{G_i}(n_i) > f(n_i)\ts,
\end{equation*}
which guarantees that $\ga_\Ga \gg f$.

Also, if $m_i$ grows sufficiently fast then the $G_i$ converge very quickly
to $\bG_\om$ and by Lemma~\ref{upper_bound_for_ball} there exists $n'_i \in \{b_j\}$
such that
\begin{equation*}
\label{eq:upperbound}
\tag{upper}
\ga_\Ga (n'_i) \leq \ga_{\bG}(n'_i) \prod_{j<i} |N_i| < g(n'_i)\ts.
\end{equation*}
This guarantees that $\ga_\Ga \ll g$.

The only thing which is left is to determine how fast the sequences $\{m_i\}$ and
$\{p_i\}$ have to grow
in order to ensure the above inequalities.
Define the the sequences $m_i$ and $p_i$ as follows:
\begin{itemize}
\item $m_1=1$ and  $H_1= \mathbf{1}$.

\item Let $n_i\in \{b_j\}$ be in integer such that
$\frac{g(n_i)}{\ga_{\bG_\omega}(n_i)} >
\prod_{j < i} |H_{j}|^{2^{m_j}}$
(such integer exists since $\{b_s\}\to \infty$ and
$\frac{g(n)}{\ga_{\bG_\omega}(n)} \to \infty$).
Define $m_{i}$ such that $\cff(m_{i})> n_i$ and that $m_i > m_{i-1}$.
This choice of $m_i$ ensures that the inequalities~(\ref{eq:upperbound})
are satisfied, because
the kernel $N_i$ of the map $G_i \to \mathbf{G}_{\om,m_i}$
is isomorphic
to the direct sum of $2^{m_i}$ copies of $H_i$ by Lemma~\ref{l:diam_of_F^i}.

\item Let $n'_i \in \{a_j\}$ be in integer such that
$$
\frac{n'_i}{\log f(n'_i)} \geq \ddd_{m_i}=\rK \cdot 2^{m_{i}}\ts,
$$
where $\ddd_i$ is the constant from Corollary~\ref{cor:weka_size_of_balls}.
Such integer exists since $\{a_i\}\to \infty$ and
$n/\log f(n) \to \infty$.
Define  $H_{i}$ to be a group together with generating set
(twisted by $\varphi^{x_{m_i+1}}$) from
Lemma~\ref{gen_set_in some_groups} of size more than
$e^{n'_i}$. Again this choice of $n'_i$ ensures that the inequalities~(\ref{eq:lowerbound})
are satisfied.
\end{itemize}

\noindent
These two (rather crude) estimates for the size of the balls in $\Gamma$
shows that conditions~(\ref{eq:lowerbound}) and~(\ref{eq:upperbound})
are satisfied. Therefore the growth function of $\Ga$ is infinitely often
larger than $f$ and infinitely often smaller than $g$.
\end{proof}

%

\smallskip

\section{Control of the upper bound}
\label{s:control}

\subsection{}
Roughly speaking, we obtain a very good control over
the upper bound by using finite groups~$H_i$ of the carefully chosen size.
We observe that Lemma~\ref{gen_set_in some_groups} gives us an
``almost continuous'' family of finite groups which can be plugged
into the construction.

Unfortunately, the growth estimates we have so far are too crude for such results.
If the sequence $m_i$ grows sufficiently fast, then the growth of the group
$\Ga$ (in certain range), is very well approximated by the growth of the group
$\Ga_i=G_i \op \bG_\om$.  This is because
\begin{align*}
\ga_\Ga(n) \geq  \ga_{\Ga_i}(n) & \quad \mbox{for all } n,\\
\ga_\Ga(n) < L_i\ga_{\Ga_i}(n) & \quad \mbox{for all } n\leq \cff(m_{i+1})
\mbox{ and } L_i = \prod_{j<i} |N_i|.
\end{align*}
The first condition follows form the observation that there are maps from
the marked group $\Ga$ to both $G_i$ and $\bG_\om$. By
Remark~\ref{re:universal} this gives a is a map onto their product $\Ga_i$.
Therefore the growth in the image is slower that the growth of $\Ga$.

The second condition is a consequence of the fact that the
ball of radius $\cff(m_{i+1})$ in $\Ga$ is the same as the ball
in the product $\bp_{j\leq i+1} G_i$ or in $\bG_\om \op \left[\bp_{j\leq i} G_i\right]$,
and that
$$
\left| \ts\ker \left(\bG_\om \op \left[\bp_{j\leq i} G_i\right] \toto \ts \Ga_i\right)\right|
\, = \, L_i\ts.
$$
If $m_{i+1}$ is very large if suffice to find $H_i$ such that the growth of
$\Ga_i$ is always bellow $f_1$ but sometimes it is above $f_2$.

\subsection{}
Below we present much better bounds on the growth in the following marked group:
$$
\Laoi(H) = F^i_\om(H) \op \bG_\om \subset F^i_\om(H) \oplus \bG_\om,
$$
which is closely related to the group $\Ga_i$ mentioned above.
%
%
The growth of the balls in $\Lambda$ is in 3 different regimes
depending on the scale.  For small radius $n < t_i$ the balls
are the same as the ball in $G_\omega$ and grow sub-exponentially.

For big radius $n > T_i = 2^i D_i\diam |H|$ the finite group $F^i_\omega(H)$
has been exhausted and the size of the ball of radius $n$
is very close to $|H|^{2^i}$ time the size of the ball in
$\bG_\omega$ and again is sub-exponential.

In the intermediate range $t_i<n < T_i$ the growth is
more complicated -- it is similar to the growth in the finite
group $H$ and therefore is ``locally'' is very close to exponential.
However, the proof of the next results requires to obtain some bounds
for this intermediate range.  As usual in such situations, understanding
the exact growth in the intermediate range is extremely difficult
and our bounds are far from optimal.
Improving these bounds will result in weakening the technical
conditions~(\ref{t:main:con:growth}) of~Theorem~\ref{t:main}
and~(\ref{l:postponed:con:technical}) in~Lemma~\ref{l:postponed}.

\subsection{}
The following technical lemma ensures that we can find the group $H$
such that the growth of the group $\Laoi(H)$ is between $f_1$ and $f_2$.
We postpone the proof until the next section.

\begin{lemma}[Main lemma]
\label{l:postponed}
Let $f_1,f_2,g: \nn \to \nn$ \ts be admissible functions which satisfy the conditions
\begin{enumi}
\item $f_1(n)/g(n)$ is increasing function, 
\item $f_1(n)> g(n)^3$ for all sufficiently large $n$,
\item $f_1(n)> f_2(n)^3$ for all sufficiently large $n$, 
\item $g(n)\geq \ga_{\bG_\om}(n)$, where $\bG_\om$ is
a Grigorchuk group of intermediate growth, 
\item
\label{l:postponed:con:technical}
$\displaystyle
\exp\left[
\frac{\log g(n)}{Cn^2} \,
f_1^*\left(\frac{n}{C\log g(n)}\right)
\right]
\. > \.
\frac{C\ts f_2^*(Cn)}{n^2}$, \ts
for any $C>0$ and sufficiently large $n=n(C)$.
\end{enumi}

\smallskip

\noindent
Then, for every~$L>0$ and all sufficiently large~$i$, one can find
a finite marked group~$H_i$, such that:
\begin{enum1}
\item $H_i$ is normally generated by $r_{x_{i+1}}$,
\item there exists $n$ such that $\ga_{\lao}(n) \. > \. f_2(n)$,
\item $f_1(n) \. > \. L\ts \ga_{\lao}(n)$ for all $n > \cff(i)$,
\end{enum1}
where $\ga_\lao \ts = \ts \ga_{\Laoi(H_i)}$.
\end{lemma}
\begin{proof}[Proof of Theorem~\ref{t:main}]
The proof is almost the same as the proof of the Theorem~\ref{t:general},
but one needs to pick the groups $H_{i}$ of the correct size.
First we pick $m_1$ such that for $n > \cff(m_1)$ we have
$$
f_1(n) > f_2(n)^3
\quad\mbox{and} \quad
f_1(n) > g_1(n) > \ga_{\bG_\om}(n)
$$
which is possible because the functions satisfy conditions (\ref{t:main:con:admissible}-\ref{t:main:con:lower}).

When choosing the depths $m_i$ one need to satisfy three conditions:
the first one is $m_i > m_{i-1}$ ensures that the groups grows;
the second one as in Theorem~\ref{t:general} is
that $\frac{g_1(n_i)}{\bG_\om(n_i)} \geq L$
for some $n_i \leq \cff(m_i)$, where
$L = \prod_{j<i}  |H_j|^{2^{m_j}}$
which guarantees that the growth of $\Ga$ will be sometimes smaller than $g_1$.
The last one is that $m_i$ is larger than the bound for $i$ in
Lemma~\ref{l:postponed} which depends on $L$. 

If $m_i$ is chosen as above then we can apply the Lemma~\ref{l:postponed}
(with $g_1$ instead of $g$)
and obtain the group $H_i$. The second property of $H_i$ implies that the growth
of $\Ga$ is larger than $f_2$ for some $n > \cff(m_i)$, and the third implies
that it is bellow $f_1$  for $\cff(m_i)< n< \cff(m_{i+1})$.

As a result we have that the growth of the group $\ga_\Ga$ is between
$f_1$ and $g_2$ for all sufficiently large~$n > \cff(m_1)$,
and is above $f_2$ and below $g_1$ at least once in each interval
$\cff(m_i)< n< \cff(m_{i+1})$, which completes the proof.
\end{proof}

\smallskip

\section{Proof of Main Lemma~\ref{l:postponed}}\label{s:lemma-postponed}

\subsection{Outline}

The following is a rough outline of the proof.
We start with some estimates of the growth
of $\Laoi$ in the intermediate range: the upper bound is coming from the
submultiplicativity of the growth functions, and the lower is based on
the growth inside $H$. It is clear that the lower bound is far from being optimal,
but we suspect that the upper on is relatively close the the optima bound.

These bounds give that (Corollary~\ref{cor:upper_bound_for_H})
that if the group $H$ is small then the growth of $\Laoi(H)$
is slower than  $f_1$ and using
Corollary~\ref{size_of_balls} this growth is faster that $f_2$
if the group $H$ is big (Corollary~\ref{cor:lower_bound_for_H}).
If the gap between $f_1$ and $f_2$ is sufficiently
large then these sets have a nontrivial intersection which implies the
existence of $H$ satisfying the requirements of the lemma.

Unfortunately, for this strategy to work one need that the gap between~$f_1$
and $f_2$ to be very big.  The reason for that, is that we are using very
crude estimates for
the sizes of balls, which does not allow us to obtain better estimates for
the growth of the group $\Laoi(H)$. In order to obtain results where the
functions $f_1$ and $f_2$ we argue by contradiction.  As a result, we only
show the existence of the group~$H$, but not an algorithm to construct it.

\subsection{Three classes of marked finite groups}
First we divide the finite groups $H$ into~$3$ classes: $\Ds$, $\Dg$ and $\Deq$,
depending how the growth of $\Laom(H)$ compares with $f_1$ and $f_2$.
If one assumes that the class $\Deq$ is empty
(and the gap between $f_1$ and $f_2$ is not too small)
then $\Ds$ is closed under products of marked groups
(Corollary~\ref{cor:closed_under_prod}) which allows us to construct
a group in $\Ds$ which is much larger than the bound in
Corollary~\ref{cor:upper_bound_for_H}.
Finally one obtains a contradiction if
the size of this group is larger than the estimate from
Corollary~\ref{size_of_balls}.

\begin{defn}
Given a marked group~$H$,
to simplify the notation denote by $\ga_\lao$ the growth function
of $\Laoi(H)$.
Let $\Dsi$ denote the set of marked groups $H$ such that
$f_2(n) > \ga_{\lao}(n)$ for all $n> \cff(i)$.  Similarly, let
$\Dgi$ denote the set of marked groups $H$ such that
$f_1(n)^{2/3} \leq \ga_{\lao}(n)$ for some~$n> \cff(i)$.
Finally, by $\Deqi$ denote the set of marked groups~$H$ such that
$f_1(n)^{2/3} > \ga_{\lao}(n)$ for all~$n > \cff(i)$, but
$f_2(n) \leq \ga_{\lao}(n)$ for some~$n$.
\end{defn}

The conclusion of Lemma~\ref{l:postponed} is equivalent to saying that $\Deqi$
is not empty when $i$ is sufficiently large, since we can guarantee
that  $f_1(n)^{1/3} > L$ for $n > \cff(i)$.

\subsection{Details: large gap}

\smallskip

\begin{lemma}
\label{l:upper_bound_for_Lambda}
Fix the group~$H$.  The growth of the function $\ga_{\lao}$
is bounded above by the function
$\gaga_{T}$ defined as follows
$$
\gaga_{T} (n)= \left\{
\begin{array}{ll}
g(n) & \mbox{for } n \leq \cff(i) \\
\exp\left(  n/ \Upsi_i\right)
                 & \mbox{for } \cff(i) \leq n \leq T \\
|H|^{2^i} g(n) & \mbox{for } n \geq T \\
\end{array}
\right.
$$
where
$\Upsi_i = \min \left\{\frac{n}{\log g(n)}
\mid \frac{\cff(i)}{2} \leq n \leq \cff(i) \right\}$.
\end{lemma}

\begin{figure}[hbt]
\begin{center}
\psfrag{T}{$T$}
\psfrag{t}{$\vt(i)$}
\psfrag{g}{\small $\log \ga_{\bG_\om}$}
\psfrag{L}{\small $\log \ga_{\lao}$}
\psfrag{h}{\small $\log |H|^{2^i} \ga_{\bG_\om}$}
\psfrag{n}{$n$}
\psfrag{U}{\small \hskip-.01cm $\Ups_T$}
\psfrag{a}{\small \hskip-1.01cm $2^i \log |H|$}
\epsfig{file=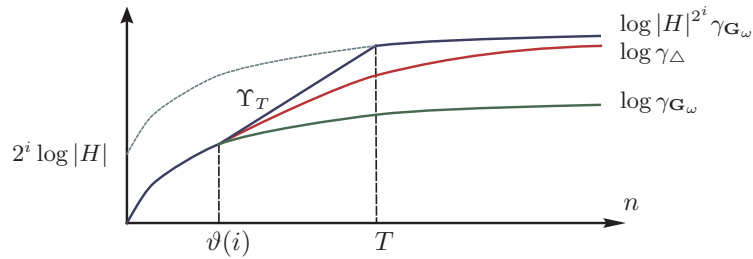,width=7.5cm}
\end{center}
\caption{The graph of functions as in Lemma~\ref{l:upper_bound_for_Lambda}.}
\label{f:estimate}
\end{figure}

\begin{proof}
The ball of radius less than $\cff(i)$ in $\Laoi(H)$ is the same as in
the group $\bG_\om$ which gives the bound for small $n$.
The kernel of
$\Laoi(H) \to \bG_\om$ has size $|H|^{2^i}$ which easily implies the bound for large $n$.

In the middle range one uses sub multiplicativity of growth functions
$
\ga_{\lao}(a+ b) \leq \ga_{\lao}(a) \cdot \ga_{\lao}(b)
$
for any $a,b$. This implies that if $n \geq m$ then
$$
\frac{\log \ga_{\lao}(n)}{n} \, \leq\,
\max_{m/2 \. \le \. s \. \le \. m} \. \frac{\log \ga_{\lao}(s)}{s}
.
$$
This inequality for $m = \cff(i)$ is equivalent to  the bound in the
middle range (see Figure~\ref{f:estimate}).
\end{proof}

The lemma implies that if the size of $H$ is very small, then the growth of
$\Laoi(H)$ is smaller then~$f_1$.

\begin{cor}
\label{cor:upper_bound_for_H}
\begin{enumi}
\item
Let $f$ be an admissible function, such that $f(n)/g(n)$ for some
$g(n) > \ga_{\bG_\om}(n)$ is increasing.  If
$$
|H| < \exp \.
 \left[
\frac{f^*(\Upsi_i)}{2^{i}\ts \Upsi_i} - \frac{\log g\ts (f^*(\Upsi_i))}{2^{i}}\right],
$$
then $f(n) > \ga_{\lao}(n)$ for all integer~$n$.

\item
Moreover, if $f$ also satisfies  $f(n)\geq g(n)^3$ for $n\geq \cff(i)$
then for
$$
|H| \. < \.
\UUU{f}{i} := \exp \ts \left[ \frac{f^*(\frac{2}{3}\Upsi_i)}{2^{i+1} \ts \Upsi_i}\right].
$$
we have that $f(n)^{2/3} > \ga_{\lao}(n)$ for all integers $n \geq \cff(i)$.
\end{enumi}
\end{cor}
\begin{proof}
For the first part, compute the point~$T$ where the graph of
$\left(f_1\right)^{2/3}$ intersects with $\exp(n/\Upsi )$.

By Lemma~\ref{l:upper_bound_for_Lambda},
if $|H| ^{2^i} \leq \frac{f_1(T)}{g(T)}$, then the growth
of $\Laoi(H)$ is slower than $\gaga_{T}$, which is less than~$f$.
%
The second part uses the estimate
$\frac{f_1(T)}{g(T)} \ts \geq \ts \bigl[f_1(T)\bigr]^{2/3}$.
\end{proof}


This following result is a strengthening of Corollary~\ref{cor:size_of_balls}.

\begin{cor}
\label{cor:lower_bound_for_H}
If $H = \PSL_2(\zz_N)$ with the generating set from
Lemma~\ref{l:gen_set_in_some_groups} and
$$
|H| \. > \. \LLL{f}{i} \.= \.
\exp\left[ \frac{f^*(\ddd'_i)}{2^i \ddd'_i}\right]\ts,
$$
then $f(n) < \ga_{\lao}(n)$ for some~$n$.
\end{cor}

\begin{proof}
By Corollary~\ref{cor:size_of_balls} $\ga_{\lao}(n) \geq \exp(n/\ddd'_i)$
for $n \leq 2^i \ddd'_i \log |H|$. For $n= f^*(\ddd_i)$ the bound is the same as $f(n)$,
but we can apply the estimate only if $|H| > \exp(n / 2^i \ddd'_i)$.
\end{proof}
\begin{proof}[Proof of Lemma~\ref{l:postponed} for large gap]
Let the functions $f_1$ and $f_2$ satisfy
\begin{equation}
\tag{big gap}
\label{eq:big_gap}
\frac{\log g(n)}{n^2}
f_1^*\left(\frac{n}{ C \log g(n)}\right)
>
\frac{C}{n^2}{f_2^*(C n)},
\end{equation}
for any constant $C$ and any sufficiently large $n$.
Notice that, up to a constants, both  $\cff(i)$ and $\ddd'_i$ are
equal to $2^i$. Substituting $\vs=2^i$ one gets
$$
\LLL{f_2}{i}  \approx
\exp \left[ \frac{C_1}{\vs^2} f_2^*(C_2\vs)\right]
$$
where $C_1$ and $C_2$ are universal constants.
Similarly,
$$
\UUU{f_1}{i} \approx
\exp \left[
\frac{\log g(\vs)}{C_3\vs^2}
f^*\left(\frac{\vs}{C_4\log g(\vs)}\right)
\right].
$$

If the functions $f_1$ and $f_2$ satisfy the equation~(\ref{eq:big_gap})
then  both $\UUU{f_1}{i}/\LLL{f_2}{i}$ and $\UUU{f_1}{i}$
then to $\infty$ as $i$ increases. Therefore, there exists~$i_0$, such that
for $i>i_0$ we have:
$$
\bigl[f_1\bigl(\cff(i)\bigr)\bigr]^{1/3} > L\ts, \quad  \UUU{f_1}{i}/\LLL{f_2}{i} > 10\ts,
\quad \UUU{f_1}{i} > 1000.
$$

Under these conditions there exists a prime $p_i=1 \mod 4$
such that $\UUU{f_1}{i} > \PSL_2(p_i) > \LLL{f_2}{i}$
because the above conditions translate to
$$ A_i > p_i > B_i$$
where $A_i/B_i > 2$ and $A_i > 13$, which allows us to apply Bertrand's postulate.

Corollary~\ref{cor:upper_bound_for_H} implies that the growth
$\ga_\laoi=\ga_{\Laoi(H_i)}$ of group $\Laoi(H_i)$, where $H_i=\PSL_2(p_i)$,
is slower than $f_1$.   Therefore
$$
L \ga_\laoi (n) \. < \. L f_1(n)^{2/3} \. < \. f_1(n)
\quad \mbox{for all } \ n \geq \cff(i).
$$

Also Corollary~\ref{cor:lower_bound_for_H}
that the growth of is not slower than $\Laoi(H_i)$, i.e.,
$$
\ga_{\laoi} (n) \. > \. f_2(n)
\quad \mbox{for some } \ n \geq \cff(i).
$$

Therefore the group $H_i = \PSL_2(p_i)$ has all necessary properties.
\end{proof}

\subsection{Details: small gap}
Unfortunately, if $f_1 \sim f_2$\ts, then the gap between the functions $f_1$
and~$f_2$ is not
sufficiently big for the above argument to work.  From this point on, we
assume that the functions $f_1,f_2: \nn \to \nn$ satisfy the following
conditions:
\begin{itemize}
\item $f_1(n) > \left(f_2(n)\right)^3$
\item $f_2(n) \geq \ga_{\bG_\om}(n)$
\end{itemize}

\smallskip

{\small
\begin{remark}
Corollary~\ref{cor:upper_bound_for_H} gives that (for a fixed $i$)
if the size of $H$ is small then $H \in \Dsi$.
In particular by Lemma~\ref{gen_set_in some_groups}
the class $\Dsi$ contains the groups $\PSL_2(p)$
for $p^3/2< \UUU{f_2}{i}$ and $p=3 \mod 4$.
\end{remark}
}

\begin{lemma}
\label{l:products_are_small}
If $H_1,H_2 \in \Dsi$, then $H_1 \op H_2$ is not in $\Dgi$.
\end{lemma}

\begin{proof}
Proposition~\ref{pro:F_preserves_products} implies that
$$
\Laoi(H_1 \op H_2) \, = \,
\Laoi(H_1) \op \Laoi(H_2)\ts.
$$
By Lemma~\ref{l:product}
$$
\ga_{\Laoi(H_1 \op H_2)}(n) \. \leq \.
\ga_{\Laoi(H_1)}(n) \cdot \ga_{\Laoi(H_2)}(n)
\. \leq \. f_2(n) \cdot f_2(n) \. < \. f_1(n)\ts,
$$
i.e., the growth of $\Laoi(H_1 \op H_2)$ is slower than $f_1$ and
the group $H_1 \op H_2$ is not in $\Dgi$.
\end{proof}

\begin{cor}
\label{cor:closed_under_prod}
If the set $\Deqi$ is empty, then $\Dsi$ is closed under~$\bp$.
\end{cor}

\begin{cor}\label{cor:prime-number}
If the set $\Deqi$ is empty, then $\Dsi$ contain $\PSL_2(\zz_N)$,
where $N$ is the product of all primes~$p=3$~mod~$4$,
such that $p^3/2 < \UUU{f_1}{i}$.  
\end{cor}

\begin{proof}
Corollary~\ref{cor:upper_bound_for_H} says that if $\PSL_2(p)$ is not inside
$\Dgi$ if $|\PSL_2(p)| = p^3/2 (1+o(1)) < \UUU{f_1}{i}$.
The previous Corollary and Remark~\ref{r:product} finish the proof.
\end{proof}

\begin{cor}
\label{cor:large_groups_in_D-}
If the set $\Deqi$ is empty, then $\Dsi$ contain $\PSL_2(\zz_N)$
$$
\log N \. \approx
\. \NNN{f_1}{i} = \frac{1}{2}\.\bigl[2 \UUU{f_1}{i}\bigr]^{1/3} \ts.
$$
\end{cor}
\begin{proof}
This follows easily from the Dirichlet theorem on the
distribution of primes (mod~$4$), and a calculation of the product of
primes given in~\cite{Ruiz}.  We omit the (easy) details.
\end{proof}


\begin{proof}[Proof of Lemma~\ref{l:postponed}]
The idea is the same as in the case of large gap between $f_1$ and $f_2$.
If one assumes that $\Deqi$ is empty the one can use Corollary~\ref{cor:large_groups_in_D-} to construct groups in $\Dsi$.

Again, up to a constants, both  $\cff(i)$ and $\ddd'_i$ are
equal to $2^i$. Substituting $\vs=2^i$ one gets
$$
\LLL{f_2}{i}  \approx
\exp \left[ \frac{C_1}{\vs^2} \, f_2^*(C_2\vs)\right],
$$
where $C_1$ and $C_2$ are universal constants.
Similarly
$$
\UUU{f_1}{i} \approx
\exp \left[
\frac{\log g(\vs)}{C_3\vs^2}
f^*\left(\frac{\vs}{C_4\log g(\vs)}\right)
\right]
\quad \mbox{and}
$$
$$
\NNN{f_1}{i} \approx
\exp \left( \exp \left[
\frac{\log g(\vs)}{C_5\vs^2}
f^*\left(\frac{\vs}{C_4\log g(\vs)}\right)
\right]\right).
$$
The condition~(\ref{l:postponed:con:technical}) implies that
$$
\NNN{f_1}{i}/\LLL{f_2}{i} \to \infty  \quad \text{and} \ \
\UUU{f_1}{i} \to \infty\ts,  \ \ \ \text{as} \ \ i \to \infty\ts.
$$
Therefore, there exists $i_0$, such that for $i>i_0$ we have:
$$\bigl[f_1\bigl(\cff(i)\bigr)\bigr]^{1/3} > L\ts, \quad \NNN{f_1}{i}/\LLL{f_2}{i} > 2\ts, \quad
\UUU{f_1}{i} > 1000.$$
However, Corollaries~\ref{cor:lower_bound_for_H} and~\ref{cor:large_groups_in_D-}
imply that if $\Deqi=\emp$, then the group $\PSL_2(N)$ is neither
in~$\Dsi$ nor in~$\Dgi$, a contradiction.
This implies that  $\Deqi$ is not empty, and therefore
there exits a group~$H$ with the desired properties.
\end{proof}

\smallskip

\section{Generalizations of the Construction}
\label{s:gen}

\subsection{} Suppose $G$ acts on a set $X$ by $H \wr_X G$; we denote
the restricted wreath product $G \ltimes \bigoplus_{i\in X} H$.
One easy modification of our construction is to use the permutation
wreath product $P \wr_X \bG_\omega$ as groups at infinity, where $P$
is a finite group and $X$ is an orbit in the action of $\bG_\om$ on
the boundary of the binary tree~$\rT_2$.  The advantage of using
these groups is that  (unlike the the groups $\bG_\omega$)
their growth rate is known in some cases, see~\cite{BE}.

\begin{thm}
\label{t:generalization}
The same as Theorem~\ref{t:main} but we use the growth of the group
$\zzz \wr_X \bG_\om$ instead the group $\bG_\om$. Here $X$ is the
boundary of the binary tree~$\rT_2$.
\end{thm}

\begin{proof}[Outline of the proof]
Here is the list of changes we need in the construction.
\begin{itemize}
\item instead of the group $\cG$, consider the free product
$\cG \ast \zz_2$,
\item modify the functors $F_i$ to include the extra generators, by adding
$G = (1;1,g)$ for every $g\in \zz_2$,
\item use $\zzz$ in place of the trivial group $\mathbf{1}$,
\item use the limit of the groups $F^i_{\omega}(\zzz)$ is $\zzz \wr_X \bG_\om$,
\item change groups $H_i$ constructed in Lemma~\ref{gen_set_in some_groups},
to contain the group~$\zzz$.
\end{itemize}

\noindent
The rest of the proof follows verbatim.  We omit the details.
\end{proof}

\subsection{}
It is easy to see that the growth types of the groups $P \wr_X G$
for fixed $G$ and different $P$ are the same if $P$ is finite and nontrivial.
Thus in the theorem one can replace $\zzz \wr_X \bG_\om$
with $P \wr_X \bG_\om$ for any finite group~$P$.

It seems possible to extend this result to wreath products of the forms
$P \wr_X \bG_\om$ where the group $P$ is not finite, but one needs a sofic
approximation (a sequence of finite groups $\{P_i\}$ which converge to~$P$)
of the group $P$ instead. The above outline need to be modified to by replacing
$\zzz$ with~$P_i$.


\subsection{}
Another easy generalization direction is to use the groups constructed in~\cite{Seg}
instead of the Grigorchuk groups $\bG_\omega$. However this will make the words
needed in Lemma~\ref{l:diam_of_F^i} and~\ref{gen_set_in some_groups}
not so explicit, but it is clear that such words exist.
In fact, as far as we are aware the
growth of these groups has not been studied and it is not clear if there
any examples of this type which are of intermediate growth.


\subsection{}
It would be interesting to analyze for which groups~$\bG$ and subexponential
functions~$f$, there exists a sequence of finite groups $G_i$ which converge
to~$\bG$ and the growth of $\bp G_i$ oscillates between $\ga_{\bG}$ and the
function~$f$.
Theorem~\ref{t:main} shows that this is possible if $\bG$ is the
Grigorchuk group $\bG_\om$ and
Theorem~\ref{t:generalization} if $\bG$ is a wreath
product of $\bG_\om$ with a finite group.
We believe that for any group $\bG$ of intermediate growth such sequence
exists, provided that the gap between the growth of $\bG$ and~$f$ is
sufficiently large:

\begin{conjecture}
\label{conjecture}
For every group $G$ of intermediate growth and a subexponential function~$f(n)$
which grows sufficiently fast (depending on $\ga_G$), there exists a sequence of
finite groups $\{G_i\}$, such that
\begin{enum1}
\item $\lim G_i = G$, and
\item the growth of $\, \Ga = \bp G_i$ oscillates between $\ga_G$ and~$f$.
\end{enum1}
\end{conjecture}

\smallskip

{\small

\section{Historical remarks and open problems}
\label{s:fin}

\subsection{} \label{ss:fin-seq}
We refer to~\cite{GP,Har1} for the introduction to groups of
intermediate growth, and to~\cite{BGS,Gri-one,Gri5,GH,GNS}
for the surveys on the subject and open problems.

Although~$\bG$ is historically first group of intermediate
growth~\cite{Gri1,Gri2}, there is now a large number of
constructions of intermediate growth \emph{branch groups}
(see~\cite{BGS}). Groups $\bG_\om$ corresponding to
infinite words $\om \in \{0,1,2\}^\infty$, were introduced by
Grigorchuk in~\cite{Gri3}.   They form a continuum family of
intermediate growth groups.  In this setting, the Grigorchuk
group $\bG = \bG_{(012)^\infty}$ corresponds to a periodic
word sequence, and is sometime called the
\emph{first Grigorchuk group}~\cite{BGS}.

\subsection{}\label{ss:fin-gri}
Let us mention that Grigorchuk's original bounds for~$\bG$ were
$\al_-(\bG) \ge 0.5$ and $\al_+(\bG) \le 0.991$.
These bounds were successively improved, with the current
records being
$$
\al_-(\bG) \ge 0.5207\ts,
\quad
\al_+(\bG) \le 0.7675\ts,
$$
where both constants correspond to solutions of certain algebraic
equations.
The bound for $\al_-$ is in~\cite{Bri} (see also~\cite{Bar2,Leo}),
and for $\al_+(\bG)$ in~\cite{Bar1,MP}.
Whether the limit $\al(\bG)$ exists remains an open problem.
However, Grigorchuk conjectures that $\al_-(\Ga) \ge \al_-(\bG)$ for
\emph{every} group of intermediate growth~\cite{Gri-one}.

\subsection{}\label{ss:fin-BE}
One of the few examples of groups of intermediate growth
where the type of the growth function is known precisely are
permutational wreath products $\Ga =P \wr_X \bG_{\om}$ where
$\om=(012)^\infty$, where $X$ is the boundary of the binary tree.
If $P$ is a finite group then
$\ga_\Ga \sim \exp(n^\alpha)$ for $\alpha = 0.7675$, if $P = \zz$
then the growth is $\ga_\Ga \sim \exp(n^\alpha \log n)$, see~\cite{BE}.

\subsection{}\label{ss:fin-free}
Free Grigorchuk group $\cG$ defined in Section~\ref{s:def},
is clearly isomorphic to a free product
$\zz_2^2 \ast \zz_2$, and thus \emph{non-amenable}.
It should not be confused with
the \emph{universal Grigorchuk group} \.
$$
\bp_\om \bG_\om \. = \. \cG/\bigcap_\om \ts\ker (\cG \toto \bG_\om)\ts,
$$
which is known to have exponential growth, and is conjectured to
be amenable \cite[$\S 8$]{Gri6}.

\subsection{} \label{ss:fin-wilson}
It is well known and easy to see~\cite{Har2}, that groups of exponential
growth cannot have oscillations:
$$
\liminf_{n\to \infty} \. \frac{\log \ga^S_{\Gamma}(n)}{n} \. = \.
\limsup_{n\to \infty} \. \frac{\log \ga^S_{\Gamma}(n)}{n}
\quad \ \text{for all} \ \<S\>=\Ga\ts.
$$
Denote this limit by $\ka(\Ga,S) > 1$.  It was recently discovered by
Wilson~\cite{Wil} that there exits groups with $\inf_S \ka(\Ga,S) = 1$
(see also~\cite{Bri1}).

\subsection{}
We conjecture that condition~(\ref{t:main:con:growth}) in the Main Theorem
can be weakened to
\smallskip

\quad (\ref{t:main:con:growth}$'$) \ $\displaystyle
\frac{\log g_1(n)}{n^2} \. \cdot \.
f_1^*\left(\frac{n}{\log g_1(n)}\right) \. \to \. \infty$

\smallskip

\noindent
If true, this would significantly weaker the conditions on the growth
of $f_1$ and~$f_2$, allowing further values of parameters in the examples
from Subsection~\ref{ss:main-ex}.

Heuristically, one expect that the growth of $\Laoi(\PSL_2(\zz_N)))$ behaves
reasonably with $N$.  This implies that if $\Dsi$ contains enough
groups then $\Deqi$ is not empty.
It is possible to prove such statement for a fixed $i$ using that
the group $\Laoi(\PSL_2(\zz)))$ which is reasonably close to a
nice arithmetic group.
However it is for from clear how to do this for all~$i$ large enough.

\subsection{}
In the context of Subsection~\ref{ss:main-sketch}, in order for this
strategy to work, infinitely many groups $H_i$ need to be nontrivial.
However, one can show that taking the limits in Section~\ref{sec_limits}
cannot possibly work if the group is not finitely presented.
The following lemma clarifies our reasoning.

\begin{lemma}
\label{l:fp}
In the context of Section~\ref{sec_limits}, if the limit group $G$ is
finitely presented, then almost all groups $N_i$ are trivial.
Consequently,  $\Gamma = G \ltimes N$ for some finite group~$N$.
\end{lemma}
\begin{proof}
Suppose that $G$ has a presentation where all realtors have length at most~$k$.
The if the ball of radius $k$ in the group $H$ coincides with the ball of radius~$k$
in $G$ then $H$ is a homomorphic image of $G$ since all defining relations of $G$
are satisfied in~$H$.
Therefore group $G_i$ are images of $G$ for big $i$ which implies that
$N_i$ are trivial (again for big~$i$).
\end{proof}

Of course, in particular, the lemma shows that the Grigorchuk groups~$\bG_\om$
of intermediate growth are \emph{not} finitely presented, a well known result
in the field~\cite{Gri5,Gri6}.  Similarly the lemma shows that Conjecture~\ref{conjecture}
implies that all groups of intermediate growth are not finitely presented
a classical old open problem~\cite{Gri-one,Gri5}.

\subsection{} \label{ss:fin-growth}
It follows from Shalom and Tao's recent extension~\cite{ST} of the
Gromov's theorem~\cite{Gro}, that every group of growth
$n^{(\log \log n)^{o(1)}}$
must be virtually nilpotent, and thus have polynomial growth.  It is a
major open problem whether this result can be extended to groups of
growth $e^{o(\sqrt{n})}$.  Only partial results have been obtained
in this direction~\cite{Gri-one} (see also~\cite{BGS,Gri4,Gri5}).

\subsection{} \label{ss:fin-sg}
The growth of groups is in many ways parallel to the study of
\emph{subgroup growth} (see~\cite{Lub,LS}).  In this case,
a celebrated construction of Segal~\cite{Seg} (see also~\cite{Neumann}),
showed that the group can have nearly polynomial growth without being
virtually solvable of finite rank.  In other words, the Shalom-Tao
extension of Gromov theorem does not have a subgroup growth analogue.
Interestingly, Segal's construction also uses the Grigorchuk
type groups, and takes the iterated permutational wreath product
of permutation groups; it is one of the motivations behind our construction.

Let us mention here that Pyber completely resolved the
``gap problem'' by describing groups with subgroup growth given
by any prescribed increasing function (within a certain range).
His proof relies on sequences of finite alternating groups
of pairwise different degrees also~\cite{Pyb2}, generalizing
a classical construction of B.~H.~Neumann~\cite{Neu}
(see also~\cite{Pyb1}).

On the other end of the spectrum, let us mention that for subgroup
growth there is no strict lower bound, i.e., for any function $f(n)$ there
is a f.g.~group where the number of subgroups of index $n$ is less than $f(n)$
infinitely often~\cite{KN,Seg}.

\subsection{}
In another variation on the group growth is the \emph{representation growth},
defined via the number $r_n(G)$ of irreducible complex representations of
dimension~$n$, whose kernel has finite index.  In this case there is again
no upper bound for the growth of~$r_n(G)$ (this follows from~\cite{KN}).
We refer to~\cite{LL} for the introduction to the subject, and
to~\cite{Cra} for recent lower bound on representation growth.
See also~\cite{Jai} and~\cite{Voll} for the zeta-function approach.

\subsection{} \label{ss:fin-alg}
The growth of algebras (rather than group) is well understood, and
much more flexibility is possible (see e.g.~\cite{Ufn}).
The results in this paper and~\cite{Bri2}, it seems, suggest
that the growth of groups can be much less rigid than previously
believed.

\subsection{} \label{ss:fin-exp}
In the proof of Lemma~\ref{gen_set_in some_groups}, using the
length $\le 10$ of standard generators in $\{a,b,c,d\}$,
one can get an explicit bound $\rK < 10 \cdot 2000$
(see~\cite[$\S 8$]{Lub-book}  and~\cite[$\S 11.2$]{HLW}).
Bounds in~\cite{BG1}, giving~$\rK'$, can also potentially be
made explicit.

Most recently, it was shown that for primes $p = 1$~mod~$4$
of positive density, \emph{all} Cayley graphs of $\PSL_2(p)$
have universal expansion, a result conjectured
for all primes~\cite{BG}.  These most general bounds have
yet to be made explicit, however.

\subsection{} \label{ss:fin-sof}
A finitely generated group $G$ called \emph{sofic} if it is a limit of
some sequence of finite groups. The existence of finitely generated
non-sofic groups is a well known open problem~\cite[$\S 3$]{Pes}.
%
%
Let us also mention that convergence of Grigorchuk groups was
also studied in~\cite{Gri3}, and a related notion of
\emph{Benjamini-Schramm convergence} for graph sequences~\cite{BS}.

\subsection{}\label{ss:fin-comparison}
As a minor but potentially important difference, let us mention
that the oscillating growth established by Brieussel~\cite{Bri2}
does not give explicit bounds on the ``oscillation times'',
while in this paper we compute them explicitly, up to some global
constants.  Since our result mostly do not overlap with those
in~\cite{Bri2}, it would be useful to quantify the former.

Interestingly, both this paper and~\cite{Bri2} have been obtained
independently and using different tools, they were both originally
motivated by probabilistic applications (see~\cite{KP}),
to the analysis of the \emph{return probability}
and the \emph{rate of escape} of a random walk on groups.
We refer to~\cite{Woe} for a general introduction to the subject.

\vskip.6cm


\noindent
\textbf{Acknowledgements.} \,
These results were obtained at the \emph{Groups and Additive Combinatorics}
workshop in June 2011.  We are grateful to Ben Green and Newton Institute
for the invitation and hospitality.
Both authors were partially supported by the NSF grants,
the second author is also partially supported by the BSF grant.

}



\vskip.8cm

\begin{thebibliography}{BNRR}

\bibitem[Bar1]{Bar1}
L.~Bartholdi, The growth of Grigorchuk's torsion group,
\emph{Int. Math. Res. Not.}~\textbf{20} (1998), 1049--1054.

\bibitem[Bar2]{Bar2}
L.~Bartholdi,
Lower bounds on the growth of a group acting on the binary rooted tree,
\emph{Int. J.~Algebra Comput.}~\textbf{11} (2001), 73--88.

\bibitem[BE]{BE}
L.~Bartholdi and A.~G.~Erschler,
Growth of permutational extensions,
{\tt arXiv:1011.5266}.

\bibitem[BGS]{BGS}
L.~Bartholdi,  R.~I.~Grigorchuk and Z.~Sunik,
Branch groups, in {\em Handbook of Algebra}, vol.~3,
North-Holland, Amsterdam, 2003, 989--1112.

\bibitem[BS]{BS}
I.~Benjamini and O.~Schramm,
Recurrence of distributional limits of finite planar graphs,
\emph{El.~J.~Probab.}~\textbf{6} (2001), Paper~23.

\bibitem[BG1]{BG1}
J.~Bourgain and A.~Gamburd,
Uniform expansion bounds for Cayley graphs of ${\rm SL}\sb 2(\mathbb{F}\sb p)$,
\emph{Ann. Math.}~{\bf 167} (2008), 625--642.

\bibitem[BG2]{BG}
E.~Breuillard and A.~Gamburd,
Strong uniform expansion in $\SL(2,p)$,
\emph{Geom. Funct. Anal.}~\textbf{20} (2010), 1201--1209.

\bibitem[Bri1]{Bri}
J.~Brieussel, \emph{Growth of certain groups of automorphisms
of rooted trees} (in French),  Doctoral Dissertation,
University of Paris~7, 2008; available at
\ts {\tt http://www.institut.math.jussieu.fr/} {\tt theses/2008/brieussel/}

\bibitem[Bri2]{Bri1}
J.~Brieussel,
Amenability and non-uniform growth of some directed
automorphism groups of a rooted tree,
\emph{Math. Z.}~\textbf{263} (2009), 265--293.

\bibitem[Bri3]{Bri2}
J.~Brieussel,  Growth behaviors in the range $e^{r^\alpha}$,
{\tt arXiv:1107.1632}.

\bibitem[Cra]{Cra}
D.~A.~Craven,
Lower bounds for representation growth,
\emph{J.~Group Theory}~\textbf{13} (2010), 873--890.

\bibitem[Ers1]{Ers1}
A.~Erschler,
Boundary behavior for groups of subexponential growth,
\emph{Ann. Math.}~\textbf{160} (2004), 1183--1210.

\bibitem[Ers2]{Ers2}
A.~Erschler,
On the degrees of growth of finitely generated groups,
\emph{Funct. Anal. Appl.}~\textbf{39} (2005), 317--320.

\bibitem[Gri1]{Gri1}
R.~I.~Grigorchuk,
On Burnside's problem on periodic groups,
{\em Funct. Anal. Appl.} \textbf{14} (1980),  41--43.

\bibitem[Gri2]{Gri2}
R.~I.~Grigorchuk,
On the Milnor problem of group growth,
{\em Soviet Math. Dokl.} \textbf{28} (1983), 23--26.

\bibitem[Gri3]{Gri3}
R.~I.~Grigorchuk,
Degrees of growth of finitely generated groups and the
theory of invariant means,
{\em Math. USSR-Izv.}  \textbf{25} (1985),  259--300.

\bibitem[Gri4]{Gri4}
R.~I.~Grigorchuk,
Degrees of growth of $p$-groups and torsion-free groups,
\emph{Math. USSR-Sb.}~\textbf{54} (1986), 185--205.

\bibitem[Gri5]{Gri-one}
R.~Grigorchik,
Solved and unsolved problems around one group, in
\emph{Infinite groups: geometric, combinatorial and dynamical aspects},
Birkh\"{a}user, Basel, 2005, 117--218.

\bibitem[Gri6]{Gri5}
R.~I.~Grigorchik,
Some topics of dynamics of group actions on rooted trees (in Russian),
\emph{Proc. Steklov Inst. Math.}~\textbf{273} (2011), 1--118.

\bibitem[Gri7]{Gri6}
R.~I.~Grigorchik,
Problem of Milnor on group growth and its consequences, preprint (2011).

\bibitem[GH]{GH}
R.~Grigorchuk and P.~de~la~Harpe,
On problems related to growth, entropy, and spectrum in group theory,
\emph{J. Dynam. Control Systems}~\textbf{3} (1997), 51--89.

\bibitem[GNS]{GNS}
R.~I.~Grigorchuk, V.~V.~Nekrashevych and V.~I.~Sushchanski\u\i,
Automata, dynamical systems, and groups,
{\em Proc. Steklov Inst. Math.}~\textbf{231} (2000), 128--203.

\bibitem[GP]{GP}
R.~Grigorchuk and I.~Pak,
Groups of intermediate growth: an introduction,
\emph{Enseign. Math.}~\textbf{54} (2008), 251--272.

\bibitem[Gro]{Gro}
M.~Gromov,
Hyperbolic groups, in \emph{Essays in group theory},
Springer, New York, 1987, 75--263.

\bibitem[Har1]{Har1} P.~de~la~Harpe,
{\em Topics on Geometric Group Theory},
University of Chicago Press, Chicago, IL, 2000.

\bibitem[Har2]{Har2} P.~de~la~Harpe,
Uniform growth in groups of exponential growth,
\emph{Geom. Dedicata}~\textbf{95} (2002), 1--17.

\bibitem[HLW]{HLW}
S.~Hoory, N.~Linial and A.~Wigderson,
Expander graphs and their applications,
\emph{Bull. Amer. Math. Soc.}~\textbf{43} (2006),  439--561.

\bibitem[Jai]{Jai}
A.~Jaikin-Zapirain,
Zeta function of representations of compact p-adic analytic groups,
\emph{J. Amer. Math. Soc.}~\textbf{19} (2006), 91--118.

\bibitem[KN]{KN}
M.~Kassabov and N.~Nikolov,
Cartesian products as profinite completions,
\emph{Int. Math. Res. Not.}~\textbf{2006},
Art.~ID~72947, 17~pp.

\bibitem[KP]{KP}
M.~Kassabov and I.~Pak,
The oscillating rate of escape of random walks
on groups, in preparation (2011).

\bibitem[LL]{LL}
M.~Larsen and A.~Lubotzky,
Representation growth of linear groups,
\emph{J.~Europ. Math. Soc.}~\textbf{10} (2008), 351--390.

\bibitem[Leo]{Leo}
Yu.~G.~Leonov,
On a lower bound for the growth function of the Grigorchuk group,
\emph{Math. Notes}~\textbf{67} (2000),  403--405.

\bibitem[Lub1]{Lub-book}
A.~Lubotzky,
\emph{Discrete groups, expanding graphs and invariant measures},
Birkh\"{a}user, Basel, 1994.

\bibitem[Lub2]{Lub}
A.~Lubotzky,
Subgroup growth, in \emph{Proc.~ICM Z\"urich},
Birkh\"{a}user, Basel, 1995, 309--317.

\bibitem[LS]{LS}
A.~Lubotzky and D.~Segal,
Subgroup growth, Birkh\"{a}user, Basel, 2003.

\bibitem[MP]{MP}
R.~Muchnik and I.~Pak, On growth of Grigorchuk groups,
{\em Int.~J.~Algebra Comput.}~\textbf{11}  (2001),  1--17.

\bibitem[Nek]{Nek}
V.~Nekrashevych,
\emph{Self-similar groups}, AMS, Providence, RI, 2005.

\bibitem[Neu1]{Neu}
B.~H.~Neumann, Some remarks on infinite groups,
\emph{J. London Math. Soc.}~\textbf{12} (1937),
120--127.

\bibitem[Neu2]{Neumann}
P.~M.~Neumann,
Some questions of Edjvet and Pride about infinite groups,
\emph{Illinois J.~Math.}~{\bf 30} (1986), 301--316


\bibitem[Pes]{Pes}
V.~G.~Pestov,
Hyperlinear and sofic groups: a brief guide,
\emph{Bull. Symbolic Logic}~\textbf{14} (2008), 449--480.

\bibitem[Pyb1]{Pyb1}
L.~Pyber, Old groups can learn new tricks,
in {\it Groups, combinatorics \& geometry},
World Sci., River Edge, NJ, 2003, 243--255.

\bibitem[Pyb2]{Pyb2}
L.~Pyber,
Groups of intermediate subgroup growth and a problem of Grothendieck,
\emph{Duke Math.~J.}~{\bf 121} (2004), 169--188.

\bibitem[Ruiz]{Ruiz}
S.~M.~Ruiz, A result on prime numbers,
\emph{Math. Gazette}~\textbf{81} (1997), 269--270.

\bibitem[Seg]{Seg}
D.~Segal, The finite images of finitely generated groups,
\emph{Proc. London Math. Soc.}~\textbf{82} (2001), 597--613.

\bibitem[ST]{ST}
Y.~Shalom and T.~Tao,
A finitary version of Gromov's polynomial growth theorem,
\emph{Geom. Funct. Anal.}~\textbf{20} (2010), 1502--1547.

\bibitem[Ufn]{Ufn}
V.~A.~Ufnarovskij,
Combinatorial and asymptotic methods in algebra,
in  \emph{Algebra~VI}, Springer, Berlin, 1995, 1--196.

\bibitem[Voll]{Voll}
C.~Voll,
Functional equations for zeta functions of groups and rings,
\emph{Ann. Math.}~\textbf{172} (2010), 1181--1218.

\bibitem[Wil]{Wil}
J.~S.~Wilson,
On exponential growth and uniformly exponential growth for groups,
\emph{Invent. Math.}~\textbf{155} (2004), 287--303.

\bibitem[Woe]{Woe}
W.~Woess, \emph{Random walks on infinite graphs and groups},
Cambridge University Press, Cambridge, 2000.

\end{thebibliography}
\end{document}